\documentclass[12pt, letterpaper,table]{amsart}
\usepackage{graphicx} 
\usepackage{url}
\usepackage{amsmath, amssymb, amsthm, mathrsfs}
\usepackage{tikz}
\usetikzlibrary{arrows.meta}  
\usepackage{cancel,multirow}

\usepackage{mathdots} 
\theoremstyle{theorem}
\newtheorem{theorem}{Theorem}
\newtheorem{lemma}{Lemma}
\newtheorem{corollary}{Corollary}
\newtheorem{conjecture}{Conjecture}
\newtheorem{quest}{Question}
\newtheorem{proposition}{Proposition}
\newtheorem{definition}{Definition}
\theoremstyle{remark}
\newtheorem{example}{Example}
\newtheorem{remark}{Remark}
\newtheorem*{remark*}{Remark}

\newcommand{\tilderot}{\widetilde{\mathsf{rot}}}
\newcommand{\rott}{\sigma}
\newcommand{\brow}{\mathsf{brow}}
\newcommand{\rtn}{\mathsf{rtn}}

\newcommand{\calD}{\mathcal{D}}

\newcommand{\calT}{\mathcal{T}}
\newcommand{\rottilde}{\widetilde{\mathsf{rot}}}

\usepackage{xcolor,color,colortbl}

\definecolor{light-blue}{HTML}{e9f1f8}
\definecolor{grayish}{RGB}{220,220,220}
\definecolor{lavender}{rgb}{0.5,0,1.0}

\newcommand{\qbinom}[2]{\left[\begin{array}{c}
     #1\\
     #2
\end{array}\right]_q}
\newcommand{\qbinomSpecial}[3]{\left[\begin{array}{c}
     #1\\
     #2
\end{array}\right]_{#3}}

\newcommand{\multichoose}[2]{\left(\!\binom{#1}{#2}\!\right)}

\newdimen\R
\usepackage{hyperref}
\hypersetup{
    colorlinks=true,
    linkcolor=lavender,
    filecolor=blue,      
    urlcolor=teal,
    citecolor=teal
}
\usepackage[nameinlink]{cleveref}

\begin{document}
\title[The Cyclic Sieving Phenomenon and frieze patterns]{The Cyclic Sieving Phenomenon and frieze patterns}

\author{Ashleigh Adams}
\email{ashleigh.adams@ndsu.edu}
\address{School of Mathematics\\North Dakota State University\\Fargo, ND 58102}

\author{Esther Banaian}
\email{esther.banaian@ucr.edu}
\address{School of Mathematics\\University of California, Riverside\\Riverside, CA 92521}



\begin{abstract}
We exhibit two instances of the cyclic sieving phenomenon - one on dissections of a polygon of a fixed type and one on triangulations of a once-punctured polygon. We use these results to give refined enumerations of certain families of frieze patterns. We also give an interpretation of finite, positive integral frieze patterns fixed under nontrivial rotations as frieze patterns from a family of orbifolds and show that these are always unitary. Finally, we give a bijection between Holm–J{\o}rgensen frieze patterns and $p$-Dyck paths, extending a recent construction of Ca{\~n}adas, Gaviria, Rios, and Espinosa, and discuss an induced rotation map on Dyck paths. Several conjectures and questions for future study are highlighted throughout the article. 
\end{abstract}

\maketitle

\section{Introduction}
\label{sec: introduction}

Sets of noncrossing diagonals on polygons (i.e. \emph{dissections} or partial triangulations) are a fundamental object in many branches of combinatorics. 
Enumerating such sets with various constraints, such as the number of diagonals used or the types of faces cut out by the diagonals, often proves to be an interesting problem, as was pointed out by Cayley over 100 years ago \cite{cayley1890partitions}. 
For a small sample of the many such enumerative formulas, see \cite{beckwith1998legendre,kirkman1857xii,PS,stanley1996polygon}. Sets with the maximal possible number of noncrossing diagonals, i.e. triangulations, are of particular interest as they are enumerated by the ubiquitous Catalan numbers.

The slogan of combinatorics is often that it is the ``study of counting'', but in practice many interesting problems in combinatorics concern the symmetries underlying a set as well. This approach even has interdisciplinary applications, such as in characterizing polymers \cite{doi2011characterization}. In the case of dissections of a polygon, it is natural to group dissections according to their cyclic equivalence classes. The size of the equivalence class tells us about the symmetry present in the original dissection.
 Reiner, Stanton, and White defined the \emph{cyclic sieving phenomenon} in order to systemize such a problem.
 \begin{definition}[\cite{reiner2004cyclic}]\label{def:CSP}
 A triple $(X,G_n,X(q))$ consisting of a set $X$ endowed with the action of a cyclic group of order $n$,  $G_n = \langle \sigma \rangle$, and a polynomial $X(q) \in \mathbb{Z}[q]$ is said to \emph{exhibit the cyclic sieving phenomenon} if $X(\zeta_n^k)$ is equal to the number of fixed points of $\sigma^k$ where $\zeta_n$ is a primitive $n$-th root of unity.    
 \end{definition}
 In particular, in order for $(X,G_n,X(q))$ to exhibit CSP, the polynomial $X(q)$ must satisfy  $X(\zeta_n^n) = X(1) = \vert X \vert$. One striking beauty of the CSP is that often $X(q)$ is a natural $q$-analogue of an enumeration formula for $X$. For an excellent survey on the topic, see \cite{sagan2010cyclic}.

 In their seminal paper, Reiner, Stanton, and White show a cyclic sieving phenomenon for sets of dissections of an $n$-gon using $k$ diagonals. 
 Our first main result (\Cref{thm:CSPAmu}) is a cyclic sieving phenomenon on sets of dissections with a specified \emph{type}, i.e., a list of the number of $i$-gons formed by the dissection for each $3 \leq i \leq n$. 
 We suspect our polynomial is, up to powers of $q$, a refinement of the polynomial given in \cite[Theorem 7.1]{reiner2004cyclic}. 
 Dennis Stanton communicated a conjectural formula relating these polynomials; see Conjecture \ref{con:Stanton}.

 Our motivation for studying dissections of fixed type is partially due to the theory of frieze patterns. 
 Outside of mathematics, a ``frieze pattern'' is a pattern in architecture or art with horizontal symmetry. 
 It is classically known that there are seven types of symmetries a frieze can exhibit. 
 Coxeter introduced a mathematical variant while studying Gauss's Pentagramma mirificum \cite{Coxeter}. 
Together with Conway, they more fully established the theory of frieze patterns in two papers, one of which listed a set of questions \cite{conway1973triangulated}, while the second provided the solutions \cite{conway1973answers}. 
Interest in frieze patterns has surged in the past 20 years since they were shown to have a connection to cluster algebras and representation theory \cite{caldero2006cluster}; see for instance \cite{SLk,FontainePlamondon,germain2023friezes,Propp}.

The celebrated result in \cite{conway1973triangulated,conway1973answers} is that there is a bijection between finite frieze patterns consisting of positive integers and triangulations of polygons (see \Cref{thm:CCFrieze}). 
Holm and J{\o}rgensen generalized this result by constructing a family of frieze patterns which are in 1-1 correspondence with $p$-angulations (i.e. dissections which cut a polygon into $p$-gons). 
These correspondences are useful as they link frieze patterns to a treasure trove of other mathematics (which we also exploit in \Cref{sec: frieze patterns and dyck paths}). 
However, frieze patterns are defined to have infinite rows.  From an enumerative standpoint, it would be more natural to count frieze patterns as living on a cylinder. That is, it would be more natural to count friezes up to a global shift of their rows. Understanding the cyclic sieving phenomenon on triangulations and more generally $p$-angulations allows us to do just that. 
This is stated in \Cref{thm:num of equiv classes of friezes of type lambda} and uses an alternate definition of the cyclic sieving phenomenon, stated in \Cref{prop:CSPVersion2}.

A triangulation of a polygon can have one of three types of rotational symmetry. Using a combinatorial interpretation of entries in a frieze pattern by Broline, Crowe, and Isaacs \cite{BCI}, in \Cref{prop:SortFriezeByGrCo}, we characterize the associated frieze patterns in terms of \emph{growth coefficients}, as in \cite{baur2019growth}. We then show that frieze patterns associated to triangulations with nontrivial rotational symmetry are in correspondence to homomorphisms on certain \emph{finite-type (generalized) cluster algebras}; see \Cref{prop:OrbifoldFriezesUnitaryOrder2} and \Cref{prop:OrbifoldFriezesUnitaryOrder3}. We remark that the former result follows directly from \cite{FontainePlamondon}. As these cluster algebras have geometric models via arcs on orbifolds, we refer to these as \emph{frieze patterns from orbifolds}. \cite[Proposition 4]{BK-FPSAC} implies that the frieze patterns considered here constitute all the finite, positive integral frieze patterns from orbifolds . 

  An \emph{infinite frieze pattern} has one boundary row of 0's and infinitely many rows extending in one direction. These also enjoy connections to to cluster theory and representation theory \cite{baur2024infinite,baur2024infinite2,GMV,Pallister}. Baur, Parsons, and Tschabold gave a geometric model for infinite frieze patterns of positive integers \cite{baurInfiniteFriezes}. In particular, if the infinite frieze pattern has periodic rows, then it can be associated to a family of triangulated annuli or a family of triangulated once-punctured discs. The fact that this correspondence is not bijective makes enumeration less clear. By considering a cyclic-sieving phenomenon on once-punctured discs, using a $q$-analogue of a formula in \cite{FontainePlamondon}, we enumerate certain families of infinite frieze patterns of integers (\Cref{thm:CountInfiniteIntegralFrieze}). We also extend the enumeration in \cite{FontainePlamondon} to $(m+2)$-angulations of a once-punctured polygon (\Cref{thm:CountMAngulation}) and find a partial cyclic sieving phenomenon in this setting (\Cref{prop:CountInfiniteTypeLambdapFrieze}). 

  This narrative began with a Catalan object, namely, triangulations. \emph{Dyck paths}, i.e. lattice paths from $(0,0)$ to $(n,n)$ which lie above the line $y=x$, are another well-studied and much loved Catalan object. Recently Ca{\~n}adas, Gaviria, Rios, and Espinosa gave an explicit connection between Dyck paths and finite, positive integral frieze patterns \cite{canadas2023coxeter}. In \Cref{prop:FrizeFromDyck}, we extend this latter construction to one between Holm-J{\o}rgensen frieze patterns and $p$-Dyck paths, a common generalization. In line with the theme of the article, we also describe the induced map on Dyck paths and $p$-Dyck paths given by rotating the corresponding dissection; see \Cref{thm: rotation of n-gon and shifted Dyck path}. 

The remainder of this paper is split into four sections, which are organized as follows. 
We begin in \Cref{sec: dissections of polygons} by introducing notation that will be built upon in the later sections. This section also includes our cyclic sieving results on dissections of polygons. \Cref{sec: m-angulations of punctured polygons} tells a parallel story in the context of dissections of punctured polygons. Unlike in the previous section, here some of the enumeration results are new. The results in these previous two sections are applied to frieze patterns in \Cref{sec: frieze patterns}, which also contains all necessary background information. All discussion involving Dyck paths is contained in \Cref{sec: frieze patterns and dyck paths}. Throughout the article, we have scattered questions and conjectures which we believe would lead towards interesting future work.

\section*{Acknowledgements}
Ashleigh Adams was supported by NSF grant DMS-2247089 and NSF Graduate Fellowship Grant No. 2034612. We thank Vic Reiner and Dennis Stanton on enlightening conversations on the interactions between different cyclic sieving polynomials present in this topic. We would also like to thank Martin Rubey and Jessica Striker for helpful comments on earlier drafts. The second author thanks Emine Y{\i}ld{\i}r{\i}m for pointing out the reference \cite{Palu}. 

\section{Dissections of Polygons}
\label{sec: dissections of polygons}
In this section, we will exhibit a cyclic-sieving phenomenon on families of dissections \newline (\Cref{thm:CSPAmu}). We will also introduce notation that will be used in future sections.

Let $P_n$ denote a polygon with $n$ vertices. We will label these vertices $v_0,v_1,\ldots,v_{n-1}$, traveling in clockwise order, and we will always treat these indices modulo $n$.
A (counterclockwise) \emph{rotation} of an $n$-gon, $P_n,$ is a map on the vertices of $P_n$ sending $v_i\mapsto v_{i-1}$ for all $0\leq i\leq n-1.$ 

A \emph{dissection} $T$ of $P_n$ is a set of non-crossing diagonals on $P_n$.  
We will refer to connected components of $P_n \backslash T$ as \emph{subgons}, and we will refer to the number of vertices of a polygon as its \emph{size}. 
Sometimes,  we will conflate a dissection with this set of subgons.  
We say the \emph{type} of a dissection $T$ is a vector $(\mu_1(T),\ldots,\mu_{n-2}(T))$ such that the number of $(i+2)$-gons in $P_n \backslash T$ is $\mu_i(T)$. 
When the type $\mu(T)$ of $T$ satisfies $\mu_i(T) = 0$ for all $i \neq m$, we call $T$ a \emph{$(m+2)$-angulation}. 

We first discuss an enumeration of all dissections of the same type. 
Let $\mu = (\mu_1,\mu_2,\ldots,\mu_n) \in \mathbb{Z}^n_{\geq 0}$. Throughout this section, we will set $k = \mu_1 + \mu_2 + \cdots + \mu_n$ and $n = \mu_1 + 2 \mu_2 + \cdots + n \mu_n$. 
By construction, it is possible to dissect $P_{n+2}$ into $\mu_i$ $(i+2)$-gons with $k-1$ non-crossing diagonals. 
Call the set of all such dissections $\mathcal{A}_\mu$ and let $a_\mu := \vert \mathcal{A}_\mu \vert$. A formula for $a_\mu$ was given as part of an exercise in \cite{goulden2004combinatorial}:
\[
a_\mu = \frac{1}{n+1} \binom{n+k}{k} \binom{k}{\mu_1,\mu_2,\ldots,\mu_n}.
\]

The numbers $a_\mu$ were also discussed in \cite{DevadossCellular,Wildberger,schuetz2016polygonal}.
Note that if $\mu$ has only one nonzero part, say $\mu_m$, then necessarily $\frac{n}{m}$ is a positive integer. Set $\ell = \frac{n}{m}$. After some manipulations, $a_\mu$ for such a $\mu$ can be seen to be a Fuss-Catalan number:\[
c_\ell^{(m)}:= a_{(0,\ldots,0,\mu_m,0,\ldots,0)} = \frac{1}{m\ell+1} \binom{(m+1)\ell}{\ell}.
\] In particular, if only the first part of $\mu$ is nonzero, then we recover the Catalan number $c_n := c_n^{(1)}$ which enumerates the triangulations of $P_{n+2}$.

We set $a_\mu(q)$ to be the naive $q$-analogue of $a_\mu$, namely, 
\[
a_\mu(q) = \frac{1}{[n+1]_q} \qbinom{n+k}{k}\qbinom{k}{\mu_1,\mu_2,\ldots,\mu_n}.
\]
We define $c_\ell^{(m)}(q)$ and $c_n(q)$ accordingly.
We will show that $a_\mu(q)$ functions as the cyclic sieving polynomial for $\mathcal{A}_\mu$. First, we compute the evaluation of $a_\mu(q)$ at appropriate roots of unity. Recall that we set $\zeta_n$ to be a primitive $n$-th root of unity.

\begin{lemma}\label{lem:Evaluate_a_mu}
Let $\mu = (\mu_1,\mu_2,\ldots,\mu_n), k = \sum_{i=1}^n \mu_i$, and $n = \sum_{i=1}^n i\mu_i$.  Let $d$ be a positive integer such that $d \vert (n+2)$ 
\begin{enumerate}
    \item If there is one index $j$ such that $\mu_j \equiv 1 \pmod{d}$ and $\mu_i \equiv 0 \pmod{d}$ for all $i \neq j$, then  \[
a_\mu(\zeta_d) = \binom{\frac{n+2}{d} + \frac{k-1}{d} - 1}{\frac{k-1}{d}} \binom{\frac{k-1}{d}}{\lfloor \frac{\mu_1}{d}\rfloor, \lfloor \frac{\mu_2}{d}\rfloor, \ldots, \lfloor \frac{\mu_n}{d}\rfloor} .
\]
\item If $d = 2$ and all $\mu_i$ are even, then 
\[
a_\mu(\zeta_{2}) = a_\mu(-1) = \binom{\frac{n+k}{2}}{\frac{k}{2}}\binom{\frac{k}{2}}{\frac{\mu_1}{2},\frac{\mu_2}{2},\ldots,\frac{\mu_n}{2}}.
\]
\item If $\mu$ does not satisfy the above conditions, then $a_\mu(\zeta_d) = 0$.
\end{enumerate}
\end{lemma}

\begin{proof}
A standard way to compute such evaluations is to use the following identity. If $g \equiv h \pmod{d}$, then

 \begin{equation}
\lim_{q \to \zeta_d} \frac{[g]_q}{[h]_q} = \begin{cases} \frac{g}{h} & g \equiv 0 \pmod{d} \\ 1 & \text{ otherwise}.
\end{cases}.\label{eq:Ratioq}
\end{equation}

Therefore, most of the work comes from counting the number of terms in the numerator and denominator which are equivalent to 0 modulo $d$.  In the set up of Cases (1) and (2), these numbers are the same and the resulting expressions are clear.

We now move on to showing Case (3), separating the proof into three cases, with this same method in mind.
If $k$ is not equivalent to 0 or 1 modulo $d$, then we immediately see
$
\mathop{\Bigl[ \begin{smallmatrix}
\scriptstyle n+k \\
\scriptstyle k
\end{smallmatrix} \Bigr]}_{\zeta_d}
= 0
$.

If $k \equiv 0$ modulo $d$, then the only way for 
$
\mathop{\Bigl[ \begin{smallmatrix}
\scriptstyle k \\
\scriptstyle \mu_1, \mu_2, \ldots, \mu_n
\end{smallmatrix} \Bigr]}_{\zeta_d}
$ 
to be nonzero is if $\mu_i \equiv 0 \pmod{d}$ for all $i$. If $d = 2$, then we are in Case (2). If $d > 2$, then it is impossible to have  $2 + \sum_{i=1}^n i\mu_i = 2+n$ be 0 modulo $d$.

Similarly, if $k \equiv 1$ modulo $d$, and $\mu$ does not satisfy the condition in the hypothesis of part 1, then 
$
\mathop{\Bigl[ \begin{smallmatrix}
\scriptstyle k \\
\scriptstyle \mu_1, \mu_2, \ldots, \mu_n
\end{smallmatrix} \Bigr]}_{\zeta_d} = 0
$.
\end{proof}

We now show that, if $d \vert (n+2)$ and $\mu$ satisfies the conditions in Case 1 or 2 of \Cref{lem:Evaluate_a_mu}, then the number of fixed points of $\mathcal{A}_\mu$ under $\frac{n}{d}$-fold rotation is equal to the evaluation $a_\mu(\zeta_d)$. We refer to such a dissection as one having \emph{$d$-fold symmetry}. In particular, 2-fold symmetry is a synonym for central symmetry. 
When we are in Case 1, we show this by explaining a way to encode each fixed point, which will make the enumeration apparent.  
Our proof is inspired by that of of \cite[Proposition~3.6]{eu2008cyclic}. By keeping track of specific sizes of subgons, our process can be made more explicit.

\begin{proposition}
\label{prop:CountingFixedFreeOrbit}
Let $\mu = (\mu_1,\mu_2,\ldots,\mu_n)$, $k = \sum_{i=1}^n \mu_i$, and $n = \sum_{i=1}^n i \mu_i$.  Let $d$ be a positive integer such that $d \vert (n+2)$. 
\begin{enumerate}
    \item  If there is exactly one value $\mu_j$ such that $\mu_j \equiv 1 \pmod{d}$ and $\mu_i \equiv 0 \pmod{d}$ for all $i \neq j$, then the number of dissections in $\mathcal{A}_\mu$ with $d$-fold symmetry  is\[
\binom{\frac{n+2}{d} + \frac{k-1}{d} - 1}{\frac{k-1}{d}} \binom{\frac{k-1}{d}}{\lfloor \frac{\mu_1}{d}\rfloor, \lfloor \frac{\mu_2}{d}\rfloor, \ldots, \lfloor \frac{\mu_n}{d}\rfloor} .
\]
\item Let $d = 2$. If all $\mu_i$ are even, then the number of dissections in $\mathcal{A}_\mu$ with $d$-fold symmetry  is \[
\binom{\frac{n+k}{2}}{\frac{k}{2}}\binom{\frac{k}{2}}{\frac{\mu_1}{2},\frac{\mu_2}{2},\ldots,\frac{\mu_n}{2}}.
\] 
\item If $d$ and $\mu$ do not satisfy either of the above conditions, then the number of dissections in $\mathcal{A}_\mu$ fixed by $d$-fold rotation is 0.
\end{enumerate}
\end{proposition}

\begin{proof}
\textbf{Item (1)} Notice our binomial  $\binom{\frac{n+2}{d} + \frac{k-1}{d} - 1}{\frac{k-1}{d}}$ is equal to the multichoose coefficient 
$\multichoose{\frac{n+2}{d}}{\frac{k-1}{d}}$. For shorthand, set $g = \frac{k-1}{d}$. The desired formula suggests that we should find a bijection between dissections in $\mathcal{A}_\mu$ with $d$-fold symmetry (call this set $\mathcal{A}^d_\mu$) and length $g$ lists of tuples $(a_1,e_1),(a_2,e_2),\ldots,(a_g, e_g)$ with the $a_i$ satisfying $0 \leq a_1 \leq a_2 \leq \cdots \leq a_g \leq \frac{n+2}{d}-1$ and where the $e_i$ form a multiset of elements from $[n]:= \{1,\ldots,n\}$ with number $i$ appearing with multiplicity $\lfloor \frac{\mu_i}{d}\rfloor$.

Given $S \in \mathcal{A}_\mu^d$, we see that $S$ is determined by its appearance in a $\frac{2\pi n}{d}$ sector, say, on the vertices $v_0,v_1,\ldots,v_{\frac{n+2}{d}-1}$. For the following, we will  consider each diagonal oriented so that the center of $P_{n+2}$ lies to its right. This gives us a well-defined notion of the ``starting point'' of a diagonal. 

There are $g = \frac{k-1}{d}$ orbits of diagonals under this rotation. Each orbit contains a unique representative whose starting point is in $V^d:= \{v_0,v_1,\ldots,v_{\frac{n+2}{d}-1}\}$. Let $a_1 \leq a_2 \leq \cdots \leq a_g$ be the indices of the starting points, listed in weakly increasing order. 
We will think of each value $a_i$ associated with the corresponding diagonal with starting point at $v_{a_i}$. 
When there are multiple diagonals with the same starting point, we will label these in counterclockwise order, so that the one with larger index is closer to the boundary edge $(v_{a_i},v_{a_i+1})$. 
Now, let $e_i + 2$ be the size of the subgon lying to the left of the $i$-th diagonal. 
Here, by ``to the left'', we are again using the reference of  the orientation given to this diagonal. Since either $d > 2$ or $d = 2$ and there is an even number of diagonals, there will be one subgon, whose size is a multiple of $d$, which will never be detected by this process since it lies to the right of all arcs. 
It follows that the list $(a_1,e_1),\ldots,(a_g,e_g)$ will satisfy the desired conditions.

We claim this map is bijective. We describe an inverse map recursively. That is, we describe how to build an element of $\mathcal{A}_\mu^d$ from such a list. 
The smallest case is $g = \frac{k-1}{d} = 0$, which is associated to the trivial dissection on an $(n+2)$-gon. 

Now, we look more generally at a list $(a_1,e_1),\ldots,(a_g,e_g)$ where in the multiset $\{e_1,\ldots,e_g\}$ the number $i+2$ appears with multiplicity $\lfloor \frac{\mu_i}{d}\rfloor$.  The key step is showing that there exists at least one value $i$ such that $a_i + e_i + 1 \leq a_{i+1}$, where we interpret $a_{g + 1} = a_1 + \frac{n+2}{d}$. 
Now, suppose for sake of contradiction that, for all $1 \leq i \leq g$, $a_i + e_i + 1 > a_{i+1}$, or equivalently $a_i + e_i \geq a_{i+1}$. 
Since $a_{g + 1}$ is defined to be $a_1 + \frac{n+2}{d}$, we also have an inequality $a_g + e_g \geq a_1 + \frac{n+2}{d}$.  If we sum all $g$ such inequalities and reduce,  we have  \[
 e_1 + \cdots + e_{g-1} + e_g \geq \frac{n+2}{d}  .
\]

Recall $j$ is the index such that $\mu_{j} \equiv 1 \pmod {d}$.  Now, we view the sum of the $e_i$ in another way,  using our initial assumptions about the vector $\mu$, 
\begin{align*}
\sum_{i=1}^g e_i &= \sum_{i=1}^n \bigg\lfloor \frac{\mu_i}{d} \bigg\rfloor i\\
&= \sum_{i \neq j} \frac{\mu_i}{d} i + \frac{j(\mu_j - 1)}{d}\\
&= \frac{1}{d} \sum_{i=1}^n \mu_i i - \frac{j}{d}\\
&= \frac{n+2}{d} - \frac{j}{d} < \frac{n+2}{d} .
\end{align*}

Therefore,  assuming $a_i + e_i + 1 > a_{i+1}$ for all $i$ led to a contradiction.  
Given a list $(a_1,e_1),\ldots,(a_g,e_g)$,  pick the smallest $i$ such that $a_i + e_i < a_{i+1}$. 
If we remove the pair $(a_i,e_i)$, we get a smaller list which has an associated dissection $S'$ in $\mathcal{A}_{\mu'}^d$ where $\mu'$ is the result of subtracting $1$ from $\mu_{e_i}$. 
Then, the dissection for the original list is the result of adding an $(e_i+2)$-gon onto $S'$ at each boundary edge in the orbit of $(v_{a_{i}},v_{a_{i}+1})$. 

\textbf{Item (2)} If the number of subgons in a dissection is even, then the number of arcs used in the dissection is odd. In order for such a dissection to be fixed by 2-fold rotation, one arc in the dissection must be fixed by the rotation. This arc must be of the form $(i, i+\frac{n}{2})$. Such an arc cuts our polygon into two $(\frac{n}{2}+1)$-gons. We see that the dissection is fixed if these two smaller subgons have the same dissection, so that the number of fixed points is $\frac{n}{2} a_{\frac{\mu}{2}}$ where $\frac{\mu}{2}:=(\frac{\mu_1}{2},\frac{\mu_2}{2},\ldots,\frac{\mu_n}{2})$. The claim follows by simplifying the expression $\frac{n}{2} a_{\frac{\mu}{2}}$.

\textbf{Item (3)} If $d > 2$, then a necessary condition for a dissection to have $d$-fold symmetry is for there to be a central $(\ell d)$-gon, implying that the number of $(\ell d)$-gons in the entire dissection is 1 modulo $d$ and the number of all other sized subgons is 0 modulo $d$. If $d = 2$, then there must either be a central $(2\ell)$-gon for $\ell > 1$ or a central diagonal. In all such cases, we see that one of the conditions of $\mu$ from (1) or (2) must be satisfied. 
\end{proof}

Combining \Cref{lem:Evaluate_a_mu} and \Cref{prop:CountingFixedFreeOrbit} gives the following immediately.

\begin{theorem}
\label{thm:CSPAmu}
The triple $(\mathcal{A}_\mu, G_{n+2}, a_\mu(q))$ exhibits the cyclic sieving phenomenon. 
\end{theorem}

\subsection{Relationship to Previous Work}

In \cite[Theorem 7.1]{reiner2004cyclic}, Reiner, Stanton, and White discussed cyclic sieving phenomenon on the set of all dissections of $P_n$ which use $k$ diagonals. The number of such dissections is $f(n,k) = \frac{1}{n+k} \binom{n+k}{k+1}\binom{n-3}{k}$, and the polynomial used their result is the naive $q$-analogue of this formula, i.e., \[
f(n,k;q):=\frac{1}{[n+k]_q}\qbinom{n+k}{k+1}\qbinom{n-3}{k}.\]

The sets $\mathcal{A}_\mu$ partition the cyclic orbits of dissections of $P_{n+2}$ using $k-1$ diagonals since the size of subgons in a dissection do not change under rotation. So it is natural to expect a relation between $f(n+2,k-1;q)$ and the family of polynomials $a_\mu(q)$ with $k = \sum_{i=1}^n \mu_i$ and $n = \sum_{i=1}^n i\mu_i$. Such a relationship was conjectured to us by Dennis Stanton \cite{DennisEmail}.

Notice each $\mu$ encodes a partition $\lambda(\mu) = (\lambda_1,\ldots,\lambda_k)$ of $n$ into $k$ parts, where $\mu_i$ is the number of parts of $\lambda$ of size $i$. Write $\lambda(\mu) = (\lambda_1 > \lambda_2 > \cdots > \lambda_k)$ in decreasing order and then set \[
b(\mu) := 2k\lambda_1 + 2(k-1) \lambda_2 + \cdots + 2 \lambda_{k-1} + \lambda_k.\]

\begin{conjecture}[Stanton's Conjecture]\label{con:Stanton}
The polynomials $f(n,k;q)$ and $a_\mu(q)$ are related by \[
q^{2n + k(k-1)}f(n+2,k-1;q) = \sum_\mu q^{b(\mu)} a_\mu(q)  
\]
where the righthand side is a sum over all $\mu \in \mathbb{Z}_{\geq 0}^n$ satisfying $\sum_{i=1}^n \mu_i = k$ and $\sum_{i=1}^n i\mu_i = n$.
\end{conjecture}

After some algebraic manipulations, one can see that Conjecture \ref{con:Stanton} can be simplified as \begin{equation}
q^{2n + k(k-1)}\qbinom{n-1}{k-1} = \sum_\mu q^{b(\mu)} \qbinom{k}{\mu_1,\ldots,\mu_n}.\label{eq:Stanton}
\end{equation}
Note that at $q=1$, \Cref{eq:Stanton} recovers a simple formula regarding the number of weak compositions of $n$.

When $k = 2$, verifying \Cref{eq:Stanton} is quite easy. We were able to inductively verify \Cref{eq:Stanton} for $k = 3$ using the following process. Let the set of partitions of $n$ into $k$ parts be denoted $P(n,k)$. There is a natural inclusion, call it $\iota$, from $P(n-1,k)$ into $P(n,k)$ given by $\iota(\lambda_1,\ldots,\lambda_k) = (\lambda_1+1,\ldots,\lambda_k)$. When $k = 3$, one can characterize the complement $P(n,3) \backslash \iota(P(n-1,3))$ as well as all $\lambda \in P(n-1,3)$ such that the number of distinct parts in $\lambda$ differs from the number in $\iota(\lambda)$. Using these observations, we can classify the difference of the right-hand side  of \Cref{eq:Stanton} for $k = 3$ and the pair $n-1,n$ and show it matches the difference of the left-hand side.

Eu and Fu generalized \cite[Theorem 7.1]{reiner2004cyclic} to \emph{$s$-divisible dissections}. These are dissections in which every subgon has size $sj+2$ for a fixed positive integer $s$; the setting of Reiner, Stanton and White's work is the $s=1$ case. Dissections into $(sj+2)$-gons form a geometric model of the faces of a generalized cluster complex in type $A$ \cite{GenClusterComplex}. 

Note that $s$-divisible dissections only exist for polygons with $sn+2$ vertices. The polynomial which Eu and Fu associated to the set of $s$-divisible dissections of $P_{sn+2}$ with $k$ diagonals is \[
G(s,n,k;q) = \frac{1}{[sn+k+2]_q} \qbinom{sn+k+2}{k+1}\qbinom{n-1}{k}.
\]

Based on Conjecture \ref{con:Stanton}, one would hope that $G(s,n,k;q)$ can also be written as a sum of polynomials $a_\mu(q)$, multiplied by appropriate $q$-powers.

\begin{quest}
Fix positive integers $s$ and $n$ and let $1 \leq k \leq n$. Is there a family of integers $b_s(\mu)$ such that \[
G(s,n,k-1;q) = \sum_\mu q^{b_s(\mu)}a_\mu(q)
\]
where the righthand side is a sum over all $\mu \in \mathbb{Z}^{sn+2}_{\geq 0}$ satisfying $\sum_{i=1}^n \mu_i = k, \sum_{i=1}^n i\mu_i = sn$, and $\mu_i > 0$ only if $s$ divides $i+2$?
\end{quest}

\section{$(m+2)$-angulations of Punctured Polygons}
\label{sec: m-angulations of punctured polygons}
Let $P_n^\bullet$ denote a disc with $n$ marked points on the boundary and one on the interior, referred to as a (once) punctured polygon. As before, we will enumerate the vertices on the boundary as $v_0,\ldots,v_{n-1}$ moving clockwise. We will refer to the marked point on the interior, i.e. the puncture, as $\bullet$.

We will refer to non-self-intersecting curves with endpoints in marked points $\{v_0,\ldots,v_{n-1},\bullet\}$ as \emph{arcs} and we will refer to maximal sets of pairwise non-crossing arcs as \emph{triangulations}. We remark that arcs are always considered up to isotopy and so saying  a pair of arcs does not cross precisely means there exist non-crossing representatives in the isotopy class of each. 

 Arcs incident to the puncture, i.e. those of the form $(v_i,\bullet)$, will be referred to as \emph{spokes}. For example,  consider the two triangulations of $P_6^\bullet$ in \Cref{fig:typeDTriangulation}. The triangulation on the left has four spokes while the triangulation on the right has one spoke. By the maximality of a triangulation, every triangulation of $P_n^\bullet$ will have at least one spoke.

Given $i \neq j$, there are two non-isotopic arcs incident to $v_i$ and $v_j$. 
Notice every arc separates $P_n^\bullet$ into a smaller (ordinary) polygon and a smaller punctured polygon. 
We let $(v_i,v_j)$ be the arc such that the vertices $v_{i+1},\ldots,v_{j-1}$ sit in the smaller polygon without a puncture.  
For example, in \Cref{fig:typeDTriangulation}, both triangulations have the arc $(v_5,v_1)$, but only the triangulation on the right contains the arc $(v_1,v_5)$. We have dashed $(v_1,v_5)$ in the triangulation on the right for emphasis.

There are two special instances of arcs between boundary vertices. 
The boundary edges are exactly $(v_i,v_{i+1})$. 
For each $0 \leq i \leq n-1$, we also have the arc $(v_i,v_i)$ which cuts out a once-punctured monogon. 
For instance, the triangulation on the right in  \Cref{fig:typeDTriangulation} contains the arcs $(v_1,v_1)$.

\begin{figure}
\centering
\begin{tabular}{cc}
\begin{tikzpicture}[scale=0.4]
    \newdimen\R
    \R=2.7cm
    
    \draw (0:\R) \foreach \x in {60,120,...,360} {  -- (\x:\R) };   

    \foreach \x/\l/\p in
     {60/{\small $v_0$}/above,
      120/{\small $v_5$}/above,
      180/{\small $v_4$}/left,
      240/{\small $v_3$}/below,
      300/{\small $v_2$}/below,
     360/{\small $v_1$}/right
     }
    \node[inner sep=1pt,circle,draw,fill,label={\p:\l}] at (\x:\R) {};

    \draw[fill=black] (0,0) circle (.22cm);

    \draw (120:\R) -- (360:\R) -- (0,0) -- cycle;
    \draw (180:\R) -- (300:\R) -- (0,0) -- cycle;

\end{tikzpicture}&
\begin{tikzpicture}[scale=0.4]
   \newdimen\R
    \R=2.7cm
    
    \draw (0:\R) \foreach \x in {60,120,...,360} {  -- (\x:\R) };   

    \foreach \x/\l/\p in
     {60/{\small $v_0$}/above,
      120/{\small $v_5$}/above,
      180/{\small $v_4$}/left,
      240/{\small $v_3$}/below,
      300/{\small $v_2$}/below,
     360/{\small $v_1$}/right
     }
    \node[inner sep=1pt,circle,draw,fill,label={\p:\l}] at (\x:\R) {};

    \draw[fill=black] (0,0) circle (.22cm);
    
    \draw(\R,0) to [out = 150, in = 90] (-0.8,0);
    \draw(\R,0) to [out = 210, in = -90] (-0.8,0);
    \draw(\R,0) -- (0,0);
    \draw(\R,0) to [out = 140, in = -30] (2.7*-0.5,2.7*0.866);
    \draw[dashed](\R,0) to [out = 220, in = 0] (0,-1.4);
    \draw[dashed](0,-1.4) to [out = 180, in = -90] (2.7*-0.5,2.7*0.866);
    \draw(\R,0) to [out = 240, in = 15] (2.7*-0.5,-2.7*0.866);
    \draw (2.7*-0.5,-2.7*0.866) to [out = 105, in = 255] (2.7*-0.5,2.7*0.866);
\end{tikzpicture}\\
\end{tabular}
    \caption{Two triangulations of $P_6^\bullet$. Arc $(v_1,v_5)$ in the triangulation on the right is dashed for emphasis.}
\label{fig:typeDTriangulation}
\end{figure}
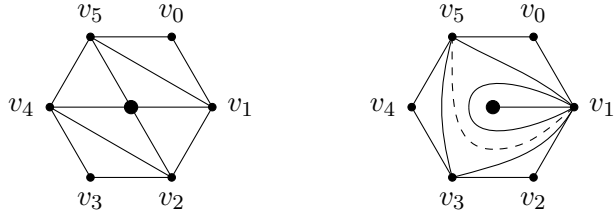

\begin{definition}
Let $m \geq 1$. An \emph{$(m+2)$-angulation} $T$ on $P_n^\bullet$ is a set of arcs which are pairwise non-crossing and such that all connected components of $P_n^\bullet \backslash T$ are $(m+2)$-gons.
\end{definition}

The latter condition implies that every $(m+2)$-angulation contains at least one spoke since otherwise there would be a connected component which contained the puncture.

\subsection{Enumeration}
\label{subsec: enumeration}
Our goal in this section is to enumerate all $(m+2)$-angulations of $P_n^\bullet$ with a fixed number of spokes. First, we determine the necessary relationship between $n$ and $m$ in order for $(m+2)$-angulations of $P_n^\bullet$ to exist.

\begin{lemma}
\label{lem:WhenCanWemangulate}
An $(m+2)$-angulation of $P_n^\bullet$ exists if and only if $n = m\ell$ for $\ell \geq 1$. If $n = m\ell$, there are exactly $\ell$ connected components in $P_n^\bullet \backslash T$ for every $(m+2)$-angulation $T$.
\end{lemma}

\begin{proof}
Every $(m+2)$-angulation must have at least one spoke. Given such a spoke, say $(v_i,\bullet)$, we see that the $P_n - (v_i,\bullet)$ is a $(n+2)$-gon. It is well-known that an $(m+2)$-angulation of $P_{n+2}$ exists exactly when $n+2 = m\ell + 2$ for some $\ell \geq 1$, and such an $(m+2)$-angulation breaks $P_{n+2}$ into $\ell$ total $(m+2)$-gons. 
\end{proof}

Fontaine and Plamondon \cite{FontainePlamondon} enumerated triangulations of $P_n^\bullet$ with a fixed number of spokes. 
Our enumeration is inspired by theirs and in particular we recover their enumeration in the triangulation case.
Fontaine and Plamondon's proof used the fact that the number of triangulations of $P_n$ which do not include any diagonals incident to a fixed vertex is a Catalan number. 
As an intermediate result, we provide a computation for the analogous situation, but instead for $(m+2)$-anguations. To that end, define $p_\ell^{(m)}$ to be the number of $(m+2)$-angulations of $P_{m\ell + 2}$ which do not include a diagonal of the form $(v_0,v_i)$ for any $2 \leq i \leq m\ell$. 

In the following, let $c^{(m)}(x) = \sum_{\ell\geq 0} c^{(m)}_\ell x^\ell$ be the generating function of the Fuss-Catalan numbers. Recall $c^{(m)}_\ell = \frac{1}{m\ell + 1} {(m+1)\ell \choose \ell}$ counts the number of $(m+2)$-angulations of a $P_{m\ell+2}$.

\begin{lemma}
\label{lem:CountSkippingOneVertex}
The generating function for the numbers $p_\ell^{(m)}$ is given by \[
p^{(m)}(x):=\sum_{\ell\geq 0} p_\ell^{(m)}x^\ell = xc^{(m)}(x)^m.
\]
\end{lemma}

\begin{proof}
Set $n = m\ell + 2$. An $(m+2)$-angulation of $P_n$ breaks the polygon into $\ell$ subgons.  If an $(m+2)$-angulation does not include any diagonals incident to $v_0$, then $v_0$ lies on a unique $(m+2)$-gon, say $Q$, and $v_1$ and $v_{n-1}$ also lie on $Q$. That is, $(v_{n-1},v_0)$ and $(v_0,v_1)$ are edges of $Q$.

The subgon $Q$ will have $m$ additional edges. These must be placed so that the resulting subgons can be $(m+2)$-angulated.
We can see the number of ways to do this as the number of ways to write $\ell-1$ as a sum of $m$ nonnegative integers, where each integer counts the number of subgons needed to $(m+2)$-angulate the subgon cut out by one edge of $Q$. In particular, we have a contribution of 0 when an edge of $Q$ coincides with an edge in $P_n$.

Given an admissible choice of edges for $Q$, yielding the weak composition $i_1 + i_2 + \cdots + i_m = \ell-1$, we see the number of ways to complete this to an $(m+2)$-angulation of $P_n$ is $c_{i_1}^{(m)} c_{i_2}^{(m)} \cdots c_{i_{m}}^{(m)}$. By considering all possible choices of $i_1,i_2,\ldots,i_{m}$, we have \begin{align*}
p_\ell^{(m)} = \sum_{i_1 + \cdots + i_{m} = \ell-1} c_{i_1}^{(m)} c_{i_2}^{(m)} \cdots c_{i_{m}}^{(m)} &= [x^{\ell-1}] c^{(m)}(x)^m.
\end{align*}
The claim now follows.

\end{proof}

Understanding the quantity $p^{(m)}_\ell$ in terms of Fuss-Catalan numbers enables us to do the desired enumeration.

\begin{theorem}
\label{thm:CountMAngulation}
Let $n = m\ell$ and let $1 \leq s \leq \ell$. The number of $(m+2)$-angulations of $P_n^\bullet$ with $s$ spokes is  $m {n+\ell -s -1 \choose n-1}$.
\end{theorem}

\begin{proof}
It is easiest to first only consider $(m+2)$-angulations of $P_n^\bullet$ which include the spoke $(v_0,\bullet)$. We must choose $s-1$ additional spokes in such a way that the polygons cut out by consecutive spokes can be $(m+2)$-angulated. Recall from \Cref{lem:WhenCanWemangulate} that any $(m+2)$-angulation of $P_n^\bullet$ will contain $\ell$ subgons. We can see our choices of valid locations for spokes as a way to write $\ell$ as a sum of $s$ positive integers, i.e., a composition of $\ell$. In particular, the composition, say $j_1 + \cdots + j_s = \ell$ is associated to a choice of spokes such that connected component cut out by the first spoke (i.e. $(v_0,\bullet)$) and second is a polygon of size $mj_1 + 2$, and so on. For one such admissible set of spokes, we then put an $(m+2)$-angulation on each smaller polygon. When doing this, we do not want to include any diagonal incident to $\bullet$ since we have already chosen all of our spokes. Therefore, we see that the number of ways to complete this choice to an $(m+2)$-angulation is $p_{j_1}^{(m)} p_{j_2}^{(m)} \cdots p_{j_s}^{(m)}$. Then, if we sum all possible choices of spokes, we have 
\[ \sum_{j_1 + \cdots + j_s = \ell} p_{j_1}^{(m)} p_{j_2}^{(m)} \cdots p_{j_s}^{(m)}
=[x^\ell] (p^{(m)}(x))^s
= [x^{\ell-s}] (c^{(m)}(x))^{ms}
\]
where the last equality follows from \Cref{lem:CountSkippingOneVertex}.  Note that since $p^{(m)}(x)$ has no constant term, summing over compositions here is equivalent to summing over weak compositions.

From \cite[Equation 7.70]{Concrete}\footnote{Note that our choice of index on the Fuss-Catalan numbers differs from the convention in this reference.}, we have that  
\begin{align*}
    [x^{\ell-s}] (c^{(m)}(x))^{ms} 
    & = \frac{ms}{(\ell-s)(m+1) + ms}{(\ell-s)(m+1) + ms \choose \ell-s}\\
    & = \frac{ms}{n + \ell - s} {n + \ell - s \choose n}\\
\end{align*}
where the second equality follows from substituting $\ell m = n$. 

So far, we have enumerated all $(m+2)$-angulations of $P_n^\bullet$ with $s$ spokes where one spoke is $(v_0,\bullet)$. We will encounter every $(m+2)$-angulation with $s$ spokes by rotating each such dissection with a spoke $(v_0,\bullet)$ $n$ times (including the trivial rotation). Doing this in fact produces every $(m+2)$-angulation $s$ times as each spoke will be incident to $v_0$ at one rotation. Therefore, the total number of $(m+2)$-angulations of $P_n^\bullet$ with $s$ spokes is
\[
    \frac{n}{s}\cdot \frac{ms}{n + \ell - s} {n + \ell - s \choose n}
   = m{n+ \ell - s -1 \choose n-1}.
\]
\end{proof}

\begin{example}
Here, we illustrate how to enumerate  5-angulations of $P_{15}^\bullet$ with 3 spokes. 

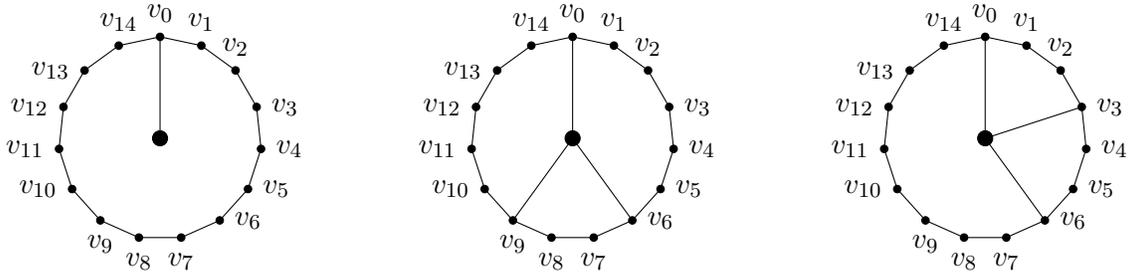
\begin{center}
\begin{tabular}{ccc}
\begin{tikzpicture}[scale=0.5]
    \newdimen\R
    \R=2.7cm
    \draw (18:\R) \foreach \x in {42,66,...,378} {  -- (\x:\R) };   
    \foreach \x/\l/\p in
     {18/{\small $v_3$}/right,
      42/{\small $v_2$}/above,
      66/{\small $v_1$}/above,
      90/{\small $v_0$}/above,
      114/{\small $v_{14}$}/above,
      138/{\small $v_{13}$}/left,
      162/{\small $v_{12}$}/left,
      186/{\small $v_{11}$}/left,
      210/{\small $v_{10}$}/left,
      234/{\small $v_9$}/below,
     258/{\small $v_8$}/below,
      282/{\small $v_7$}/below,
      306/{\small $v_6$}/right,
      330/{\small $v_5$}/right,
      360-6/{\small $v_4$}/right}
    \node[inner sep=1pt,circle,draw,fill,label={\p:\l}] at (\x:\R) {};
    \draw[fill=black] (0,0) circle (.2cm);
    \draw (90:\R) -- (0,0);
\end{tikzpicture}&
\begin{tikzpicture}[scale=0.5]
    \newdimen\R
    \R=2.7cm
    \draw (18:\R) \foreach \x in {42,66,...,378} {  -- (\x:\R) };   
    \foreach \x/\l/\p in
     {18/{\small $v_3$}/right,
      42/{\small $v_2$}/above,
      66/{\small $v_1$}/above,
      90/{\small $v_0$}/above,
      114/{\small $v_{14}$}/above,
      138/{\small $v_{13}$}/left,
      162/{\small $v_{12}$}/left,
      186/{\small $v_{11}$}/left,
      210/{\small $v_{10}$}/left,
      234/{\small $v_9$}/below,
     258/{\small $v_8$}/below,
      282/{\small $v_7$}/below,
      306/{\small $v_6$}/right,
      330/{\small $v_5$}/right,
      360-6/{\small $v_4$}/right}
    \node[inner sep=1pt,circle,draw,fill,label={\p:\l}] at (\x:\R) {};
    \draw[fill=black] (0,0) circle (.2cm);
    \draw (90:\R) -- (0,0);
    \draw(306:\R) -- (0,0);
    \draw(234:\R) -- (0,0);
\end{tikzpicture}&
\begin{tikzpicture}[scale=0.5]
    \newdimen\R
    \R=2.7cm
    \draw (18:\R) \foreach \x in {42,66,...,378} {  -- (\x:\R) };   
    \foreach \x/\l/\p in
     {18/{\small $v_3$}/right,
      42/{\small $v_2$}/above,
      66/{\small $v_1$}/above,
      90/{\small $v_0$}/above,
      114/{\small $v_{14}$}/above,
      138/{\small $v_{13}$}/left,
      162/{\small $v_{12}$}/left,
      186/{\small $v_{11}$}/left,
      210/{\small $v_{10}$}/left,
      234/{\small $v_9$}/below,
     258/{\small $v_8$}/below,
      282/{\small $v_7$}/below,
      306/{\small $v_6$}/right,
      330/{\small $v_5$}/right,
      360-6/{\small $v_4$}/right}
    \node[inner sep=1pt,circle,draw,fill,label={\p:\l}] at (\x:\R) {};
    \draw[fill=black] (0,0) circle (.2cm);
    \draw (90:\R) -- (0,0);
    \draw(18:\R) -- (0,0);
    \draw(306:\R) -- (0,0);
\end{tikzpicture}\\
\end{tabular}
\end{center}

Since $15 = 3\cdot 5$, when we begin with a spoke $(v_0,\bullet)$, the admissible ways to add two more spokes correspond with the compositions of 5 with 3 parts. These are $2+2+1, 2+1+2, 1+2+2, 3+1+1, 1+3+1$ and $1+1+3$. For example, the composition $2+1+2$ corresponds to placing spokes $(v_6,\bullet)$ and $(v_9,\bullet)$, as in the middle above, while the composition $1+1+3$ corresponds to placing spokes $(v_3,\bullet)$ and $(v_6,\bullet)$, as on the right above. 

There are $(p_2^{(5)})^2 p_1^{(5)} = 9$ ways to complete the middle picture to a 5-angulation without adding more spokes and similarly there are $p_3^{(5)} (p_1^{(5)})^2 = 15$ ways to complete the right picture to a 5-angulation with 3 spokes.
As prescribed in the proof of \Cref{thm:CountMAngulation}, we see there are $27+45 = 72$ total 5-angulations of $P_{15}^\bullet$ with 3 spokes which include $(v_0,\bullet)$, and overall there are $360$ total 5-angulations of $P_{15}^\bullet$.
\end{example}

\begin{corollary}
\label{cor:CountAllMangulationPunctured}
Let $n = m\ell$ for $m,\ell \geq 1$. The number of $(m+2)$-angulations of $P_n^\bullet$ is $m{n + \ell - 1 \choose n} = {n + \ell - 1 \choose \ell}$.
\end{corollary}

\begin{proof}
The number of spokes in an $(m+2)$-angulation of $P_n^\bullet$ is at least 1 and at most $\ell$. We reach our first expression by summing the number of $(m+2)$-angulations with each number of spokes, using \Cref{thm:CountMAngulation} and applying the ``Hockey-Stick identity,'' $\sum_{i=a}^b {i \choose a} = {b+1 \choose a+1}$. The second expression follows from a few algebraic manipulations.
\end{proof}

\subsection{Cyclic Sieving}

Here, we will show that the natural $q$-analogue of our counting formula in \Cref{thm:CountMAngulation} often works as a cyclic sieving polynomial for the set of $(m+2)$-angulations of $P_n^\bullet$.

To this end, let $\mathcal{T}_{n,s}^{\bullet,(m)}$ denote the set of $(m+2)$-angulations of $P_n^\bullet$ with $s$ spokes. Naturally, $\mathcal{T}_{n,s}^{\bullet,(m)}$ is only nonempty if $m$ divides $n$ and $1 \leq s \leq \frac{n}{m}$. Given such a value $s$, define $t_{n,s}^{(m)} = m{n + \frac{n}{m} - s -1 \choose n-1}$ so that $t_{n,s}^{(m)}$ is the cardinality of $\mathcal{T}_{n,s}^{\bullet,(m)}$. We also define \[
t_{n,s}^{(m)}(q) := m\qbinom{n + \frac{n}{m}-s-1}{n-1}.
\]

The group $G_n = \langle \sigma \rangle $ of rotations defined in \Cref{sec: dissections of polygons} also acts naturally on $\mathcal{T}_{n,s}^{\bullet,(m)}$. Our goal now is to show that the polynomial $t_{n,s}^{(m)}(q)$ functions as a cyclic sieving polynomial for many $n$-th roots of unity. The proof of \Cref{prop:CSPPuncturedPolygon} will use the following computations.

\begin{lemma}
\label{lem:q-Calculation-Dn}
Let $d,k$, and $n$ be positive integers.
\begin{enumerate}
\item If $n \equiv k \pmod{d}$, then 
$
\mathop{\Bigl[ \begin{smallmatrix}
\scriptstyle n \\
\scriptstyle k
\end{smallmatrix} \Bigr]}_{q=\zeta_d}
=  \Bigl( \begin{smallmatrix}
\scriptstyle \lfloor \frac{n}{d} \rfloor \\
\scriptstyle \lfloor \frac{k}{d} \rfloor
\end{smallmatrix} \Bigr)
$.
\item If $k \equiv -1 \pmod{d}$ and $n \not\equiv -1 \pmod{d}$, then 
$
\mathop{\Bigl[ \begin{smallmatrix}
\scriptstyle n \\
\scriptstyle k
\end{smallmatrix} \Bigr]}_{q=\zeta_d} = 0
$.
\end{enumerate}
\end{lemma}

\begin{proof}
Each statement can be shown using \Cref{eq:Ratioq}. 
For instance, in Case (1), we have \[
\lim_{q \to \zeta_d} 
\Bigl[ \begin{smallmatrix}
\scriptstyle n \\
\scriptstyle k
\end{smallmatrix} \Bigr]_{q}
= \frac{(\lfloor \frac{n}{d} \rfloor) (\lfloor \frac{n}{d} \rfloor - 1) \cdots (\lfloor \frac{n}{d} \rfloor - \lfloor \frac{k}{d} \rfloor +1 ) }
{(\lfloor  \frac{k}{d}\rfloor) (\lfloor \frac{k}{d} \rfloor - 1) \cdots 1} 
= {\lfloor \frac{n}{d} \rfloor \choose \lfloor \frac{k}{d} \rfloor},
\]
immediately seeing the desired conclusion. Similarly, Case (2) can be shown by observing that there are more factors of the minimal polynomial of $\zeta_d$ in the numerator than in the denominator of the expression \[
\frac{[n][n-1] \cdots [n-k+1]}{[k][k-1] \cdots [1]}.
\]
\end{proof}

\begin{proposition}
\label{prop:CSPPuncturedPolygon}
Let $m$ and $\ell$ be two positive integers and set $n = m\ell$. Let $1 \leq s \leq \ell$, and let $d$ be a divisor of $n$. If $s$ and $\ell$ are not equivalent to the same nonzero number mod $d$, then $t_{n,s}^{(m)}(\zeta_d)$ is equal to the number of $(m+2)$-angulations in $\mathcal{T}_{n,s}^{\bullet,(m)}$ which are fixed by $\sigma^{\frac{n}{d}}$.
\end{proposition}

\begin{proof}
Let $d \vert n$. Let $T$ be an $(m+2)$-angulation in $\mathcal{T}_{n,s}^{\bullet,(m)}$ which is fixed by $\sigma^{\frac{n}{d}}$. In order for $T$ to exist, we need $d \vert s$ and $d \vert \ell$. The latter condition is necessary when we remember from \Cref{lem:WhenCanWemangulate} that $\ell$ counts the number of $(m+2)$-gons cut out by any element of $\mathcal{T}_{n,s}^{\bullet,(m)}$.

Since $T$ is fixed by $\sigma^{\frac{n}{d}}$, it is completely determined by any $2\pi/d$ sector, such as $v_0,v_1,\ldots,v_{\frac{n}{d}-1}$. If we cut out the sector formed by $v_0,\ldots,v_{\frac{n}{d}-1}, v_{\frac{n}{d}}$ and identifying $v_0$ and $v_{\frac{n}{d}}$, we have an element of $\mathcal{T}_{\frac{n}{d},\frac{s}{d}}^{\bullet,(m)}$. This process in fact gives a bijection between the elements of $\mathcal{T}_{n,s}^{\bullet,(m)}$ fixed by $\sigma^{\frac{n}{d}}$ and $\mathcal{T}_{\frac{n}{d},\frac{s}{d}}^{\bullet,(m)}$. Therefore, the number of fixed points of $\sigma^{\frac{n}{d}}$ in $\mathcal{T}_{n,s}^{\bullet,(m)}$ is exactly $t_{\frac{n}{d},\frac{s}{d}}^{(m)} = m{\frac{n}{d} + \frac{\ell}{d} - \frac{s}{d}-1 \choose \frac{n}{d}-1}$.

Now, we turn to evaluating $t_{n,s}^{(m)}(q)$ at $q=\zeta_n^{\frac{n}{d}}$, or equivalently, $ \zeta_d$. First, suppose we have $d \vert s$ and $d \vert \ell$, so that there are $d$-fixed points in $\mathcal{T}_{n,s}^{\bullet,(m)}$. Then, since $n+\ell -s - 1 \equiv n-1 \pmod{d}$, we can use \Cref{lem:q-Calculation-Dn} to show \[
t_{n,s}^{(m)}(\zeta_d) = \qbinomSpecial{n + \ell - s - 1}{n-1}{q = \zeta_d} = {\lfloor \frac{n+\ell - s- 1}{d}\rfloor  \choose \lfloor\frac{n-1}{d} \rfloor}=  {\frac{n}{d} + \frac{\ell}{d} - \frac{s}{d} - 1 \choose \frac{n}{d}-1},
\]
as desired.

Suppose now that at least one of $\ell$ or $s$ is not equivalent to 0 modulo $d$. Moreover, assume $\ell \not\equiv s \pmod{d}$, which implies $n+\ell-s - 1 \not\equiv -1 \pmod{d}$. Since $n-1 \equiv -1 \pmod{d}$, by \Cref{lem:q-Calculation-Dn}, $t_{n,s}^{(m)}(\zeta_d) = 0$.
\end{proof}

Let $t_n(q) = \sum_{s=1}^n t_{n,s}^{(1)}(q)$.  
By  \Cref{cor:CountAllMangulationPunctured},  $t_n(q)$ is a $q$-analogue of ${2n-1\choose n}$.  
Let $\mathcal{T}_n^\bullet$ denote the set of all triangulations of $P_n^\bullet$.  
Note that if $m = 1$,  the $\ell$ in the statement of \Cref{prop:CSPPuncturedPolygon} is equal to $n$ and thus will never be equivalent to a nonzero number modulo $d$, a divisor of $n$. 
Thus, the following is immediate.
 
\begin{proposition}
\label{prop:CSPPuncturedPolygonTriangulationsOnly}
The triple $(\mathcal{T}_n^\bullet, G_n,t_n(q))$ exhibits the cyclic sieving phenomenon.
\end{proposition}

\begin{remark}
\Cref{prop:CSPPuncturedPolygonTriangulationsOnly} bears similarity to the work of Eu and Fu \cite{eu2008cyclic} who exhibited cyclic sieving phenomena for generalized cluster complexes. This is similar because triangulations of $P_n^\bullet$ model clusters in a type $D$ cluster algebra and hence these will correspond to facets of the (ordinary) cluster complex. We note that Eu and Fu use instead a model given in \cite{CAII} which uses colored diagonals on an ordinary polygon. Our result is not a special case of that in \cite{eu2008cyclic} as the cyclic actions differ. In particular, the action in Eu and Fu's work, which resembles Auslander-Reiten translation, has order $2n$ when $n$ is odd whereas our action is always order $n$. 
\end{remark}

When $s \equiv \ell \pmod{d}$ and these are not divisible by $d$, then $t_{n,s}^{(m)}(\zeta_d) \neq 0$ even though there are no triangulations fixed by $\sigma^{\frac{n}{d}}$ in $\mathcal{T}_{n,s}^{\bullet,(m)}$. 
It is not clear what $t_{n,s}^{(m)}(\zeta_d)$ counts in this case. 
For example, given if $n = 12$ and $m = 3$, implying $\ell = 4$, and $s = 1$, then we have 
$
t_{12,1}^{(3)}(\zeta_3) 
= 3
\mathop{\Bigl[
\begin{smallmatrix}
\scriptstyle 14 \\
\scriptstyle 11
\end{smallmatrix}
\Bigr]
}_{q=\zeta_3}
= 3 \binom{4}{3} = 12
$. 
Meanwhile, $\mathcal{T}_{12,1}^{\bullet,(3)}$ consists of $5$-angulations of $P_{12}^\bullet$ with one spoke, and such $5$-angulations never have nontrivial rotational symmetry.
Therefore, for general $m$, our analogue of \Cref{prop:CSPPuncturedPolygon} must avoid such behavior.

\begin{proposition}
\label{prop:CSPPuncturedPolygonm-angulation}
Let $n,m,$ and $s$ be three positive integers such that $m$ divides $n$ and for all divisors $d$ of $n$, if $s \equiv \frac{n}{m} \pmod{d}$, then $s \equiv 0 \pmod{d}$. The triple $(\mathcal{T}_{n,s}^{\bullet,(m)}, G_n, t_{n,s}^{(m)})$ exhibits the cyclic sieving phenomenon.  
\end{proposition}

An example of a triple of positive integers satisfying the conditions of \Cref{prop:CSPPuncturedPolygonm-angulation} is $n = 12, m = 3, s = 3$.

\begin{quest}
Given $n = m\ell$, $d$ a divisor of $n$, and $1 \leq s \leq \ell$ such that $s \equiv \ell \not\equiv 0 \pmod{d}$, is there a natural subset of $\mathcal{T}_{n,s}^{\bullet,(m)}$ with cardinality $t_{n,s}^{\bullet,(m)}(\zeta_d)$?
\end{quest}

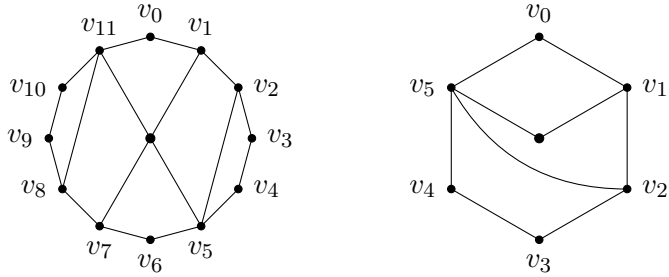
\begin{figure}
\centering
\begin{tabular}{cc}
\begin{tikzpicture}[scale=0.5]
    \newdimen\R
    \R=2.7cm
    \draw (0:\R) \foreach \x in {30,60,...,360} {  -- (\x:\R) };   
    \foreach \x/\l/\p in
     {0/{\small $v_3$}/right,
      30/{\small $v_2$}/right,
      60/{\small $v_1$}/above,
      90/{\small $v_0$}/above,
      120/{\small $v_{11}$}/above,
      150/{\small $v_{10}$}/left,
      180/{\small $v_9$}/left,
     210/{\small $v_8$}/left,
      240/{\small $v_7$}/below,
      270/{\small $v_6$}/below,
      300/{\small $v_5$}/below,
      330/{\small $v_4$}/right}
    \node[inner sep=1pt,circle,draw,fill,label={\p:\l}] at (\x:\R) {};
    \draw[fill=black] (0,0) circle (.12cm);
    \draw(60:\R) -- (0,0);
    \draw(60+180:\R) -- (0,0);
    \draw(300:\R) -- (0,0);
    \draw(120:\R)--(0,0);
    \draw(30:\R) -- (300:\R);
    \draw(210:\R) -- (120:\R);
\end{tikzpicture}&
\begin{tikzpicture}[scale=0.5]
    \newdimen\R
    \R=2.7cm
    \draw (30:\R) -- (90:\R) -- (150:\R) -- (210:\R) -- (270:\R) -- (330:\R) -- (30:\R);
    \foreach \x/\l/\p in
     {30/{\small $v_1$}/right,
      90/{\small $v_0$}/above,
      150/{\small $v_5$}/left,
     210/{\small $v_4$}/left,
      270/{\small $v_3$}/below,
      330/{\small $v_2$}/right}
    \node[inner sep=1pt,circle,draw,fill,label={\p:\l}] at (\x:\R) {};
    \draw[fill=black] (0,0) circle (.12cm);
    \draw(30:\R) -- (0,0);
    \draw(150:\R) -- (0,0);
    \draw(.867*2.7cm,-.5*2.7cm) to [out=180, in = 300] (-.867*2.7cm,.5*2.7cm);
\end{tikzpicture}\\
\end{tabular}
\caption{Associating a smaller $4$-angulation to a $4$-angulation with 2-fold symmetry, as in the proof of \Cref{prop:CSPPuncturedPolygon}.}\label{fig:SymmetricMangulation}
\end{figure}

\section{Frieze Patterns}
\label{sec: frieze patterns}

We now seek to apply our cyclic sieving results to frieze patterns,  combinatorial objects with connections to cluster theory,  representation theory,  and geometry. Several lovely surveys about frieze patterns have been recently written, such as \cite{BaurSurvey,morier2015coxeter,PresslandSurvey}.

\begin{definition}
\label{def:FriezePattern}
A \emph{frieze pattern of width $n$} over an integral domain $R$ is a function $F: \{(i,j) \in \mathbb{Z} \times \mathbb{Z}: 0 \leq i-j \leq n\} \to R$ satisfying
\begin{enumerate}
    \item $F(i,i) = F(i,n+i) = 0$ and $F(i,i+1) = F(i,n+i-1) = 1$ for all $i\in \mathbb{Z}$, and
    \item for all $i \leq j$, $F(i-1,j)F(i,j+1) - F(i,j) F(i+1,j+1)= 1$.
\end{enumerate}

An \emph{infinite frieze pattern} over a ring $R$ is a function $F: \{(i,j) \in \mathbb{Z} \times \mathbb{Z} : 0 \leq i-j\} \to R$ satisfying
\begin{enumerate}
    \item[(1')] $F(i,i) = 0$ and $F(i,i+1) = 1$ for all $i\in \mathbb{Z}$, and
    \item[(2)] for all $i \leq j$, $F(i-1,j)F(i,j+1) - F(i,j) F(i+1,j+1)= 1$.
\end{enumerate} 
\end{definition}

We remark that width $n$ frieze patterns here would often be referred to as having width $n-3$ in other articles.

The values of the frieze pattern are usually depicted as an array with every other row shifted as below, and we will conflate the data of $F$ with such arrays.
When drawn this way, condition in (2) becomes the rule that, in every diamond, the product of the horizontal entries is one more than the product of the vertical (most simply,  East $\times$ West - North $\times$ South = 1). 
Note that in this depiction, the rows are of the form $\big(F(i,i+k)\big)_{i \in \mathbb{Z}}$ for a fixed value $k$.

\begin{center}
\scalebox{0.9}{
$\begin{array}{ccccccccccccccccccccccccc}
 \ldots&0&&0&&0&&0&&0&&\ldots\\
 \ldots& &1&&1&&1&&1&&&\ldots\\
 \ldots&F(-1,1)&&F(0,2) &&F(1,3)&& F(2,4)&&F(3,5)&&\ldots\\
  \ldots&&F(-1,2)&&F(0,3)&&F(1,4)&&F(2,5)&&&\ldots\\
 \ldots  & F(-2,2)&&F(-1,3)&&F(0,4)&&F(1,5)&&F(2,6)&&\ldots\\
 \ldots &&F(-2,3)&&F(-1,4)&&F(0,5)&&F(1,6)&&&\ldots\\
 &&&\iddots&&\vdots&&\ddots&&&&&&
 \end{array}$}
 \end{center}

For explicit examples of a frieze pattern,  see  \Cref{ex:BijectionsToFPs} and \Cref{ex:InfiniteFriezePatterns}.

We introduce more terminology associated to frieze patterns. These apply to both finite frieze patterns (i.e. frieze patterns with a positive integral width) and infinite frieze patterns. In this article, all infinite frieze patterns will be assumed to be \emph{periodic}, i.e. each row will be periodic. 

\begin{definition}\label{def:BasicFriezeTerms}
Let $F$ be a frieze pattern over a ring $R$.
\begin{enumerate}
\item The row $\big(F(i,i+2)\big)$ is called the \emph{quiddity row}. 
\item If all $3 \times 3$ diamonds in $F$ form a matrix with determinant 0, we say $F$ is \emph{tame}.
\end{enumerate}
\end{definition}
One can check that all examples of frieze patterns provided here are tame. 

The unimodular rule (i.e. part (2) of \Cref{def:FriezePattern}) implies that a frieze pattern is determined by its quiddity row. Indeed, the relationship between any value of a frieze pattern and the quiddity row can be made explicit. Define a family of multivariate polynomials $K_n(x_1,\ldots,x_n)$ by $K_0 = 1$ and $K_n(x_1,\ldots,x_n) = x_n K_{n-1}(x_1,\ldots,x_{n-1}) + K_{n-2}(x_1,\ldots,x_{n-2})$. The polynomials $K_n(\mathbf{x})$ are known as \emph{continuants} and can also be defined as a determinant of a tridiagonal matrix. THe following comes from  \cite[Question 18]{conway1973triangulated,conway1973answers}.

\begin{lemma}\label{lem:Kontinuant}
Given a frieze pattern $F$ and integers $i < j$,  \[
F(i,j) = K_{j-i-1}(F(i,i+2),F(i+1,i+3),\ldots,F(j-2,j)).
\]
\end{lemma}

\subsection{Frieze Patterns from Dissections}
Frieze patterns were introduced by Coxeter in \cite{Coxeter}.  
Shortly after their first appearance, Conway and Coxeter characterized finite frieze patterns over $\mathbb{Z}$ with positive entries, i.e. positive integral frieze patterns \cite{conway1973triangulated,conway1973answers}.  
Holm and J{\o}rgensen generalized this result to $(m+2)$-angulations \cite{holm2020p}. 
We first define frieze patterns associated to any dissection of a polygon and then refine to these cases.  
Given $p$ a positive integer, set $\lambda_{p}:= 2\cos(\pi/p)$. Let $\overline{i}$ denote $i$ modulo $n$. 

\begin{definition}
\label{def:FriezePatternsFromDissections}
Let $T$ be a dissection of $P_{n+3}$.  
For each vertex $v_j$ and all $p \geq 3$, let $\mu_p(j)$ denote the number of $p$-gons of $T$ incident to $v_j$.  
We define $F_T$, the frieze pattern associated to $T$, to be the frieze pattern with quiddity row given by setting $F(i-1,i+1) = \sum_{p \geq 3} \mu_p(\overline{i}) \lambda_p$.
\end{definition}

See \Cref{ex:BijectionsToFPs} for two frieze patterns associated to dissections. 
This construction is an extension of the classical bijection between positive integral frieze patterns and triangulations of polygons. 

\begin{theorem}[\cite{conway1973triangulated,conway1973answers}]
\label{thm:CCFrieze}
Every positive integral frieze patterns of width $n$ is uniquely associated to a triangulation of $P_{n}$.  
\end{theorem}

Holm and J{\o}rgensen give a parallel theorem for general $(m+2)$-angulations.  The corresponding frieze patterns are said to be of \emph{type $\Lambda_{m+2}$}. A frieze pattern is of type $\Lambda_{m+2}$ if its quiddity row consists of positive integral multiples of $\lambda_{m+2}$.  

\begin{theorem}[\cite{holm2020p}]
\label{thm:HJ}
Every frieze pattern of type $\Lambda_{m+2}$ and width $n$ is uniquely associated to an $(m+2)$-angulation of $P_{n}$. 
\end{theorem}

No similar characterization of frieze patterns from general dissections is known. However,  Holm and J{\o}rgensen showed that these always the expected width.

\begin{proposition}[\cite{holm2020p}]
\label{prop:WidthOfFPFromDissection}
If $F_T$ is a frieze pattern associated to a dissection $T$ of $P_{n}$, then $F_T$ has width $n$.
\end{proposition}

\begin{example}
\label{ex:BijectionsToFPs}
First,  we illustrate \Cref{thm:CCFrieze} by giving an example of a triangulated hexagon and  the corresponding frieze pattern with width 3. We label each vertex $v_i$ with $F(i-1,i+1)$.

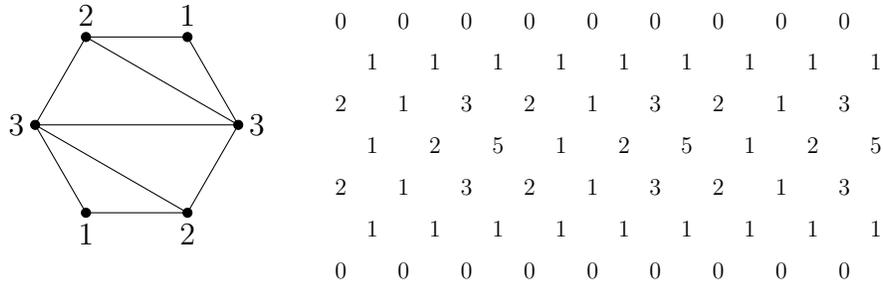
\begin{center}
  \raisebox{-.4\totalheight}{\begin{tikzpicture}[scale=0.5]
    \newdimen\R
    \R=2.7cm
    \draw(0:\R) node[right]{3} -- (60:\R)node[above]{1} -- (120:\R)  node[above]{2}-- (180:\R) node[left]{3} -- (240:\R) node[below]{1}-- (300:\R) node[below]{2}-- (360:\R);
    \draw(-60:\R) -- (180:\R) -- (0:\R) -- (120:\R);
\draw[fill=black] (0:\R) circle (.12cm);
\draw[fill=black] (60:\R) circle (.12cm);
\draw[fill=black] (120:\R) circle (.12cm);
\draw[fill=black] (180:\R) circle (.12cm);
\draw[fill=black] (240:\R) circle (.12cm);
\draw[fill=black] (300:\R) circle (.12cm);
\end{tikzpicture}}
 \scalebox{0.75}{$\begin{array}{cccccccccccccccccccc}
   &&0&&  0&&  0&&  0&&  0&&  0&&  0&&  0&&0&\\
  &&&  1&&  1&&  1&&  1&&  1&&  1&&  1&&1&&1\\
 &&2&&1&&3&&  2&&  1&&  3&&  2&&1&&3&\\
  &&&1&&  2&&  5&& 1&&  2&&  5&&1&&2&&5\\
 &&2&&1&&3&&  2&&  1&&  3&&  2&&1&&3&\\
   &&&  1&&  1&&  1&&  1&&  1&&  1&&  1&&1&&1\\
     &&0&&  0&&  0&&  0&&  0&&  0&&  0&&  0&&0&\\
 \end{array}$}
 \end{center}
 
 Now,  we turn to the setting of \Cref{thm:HJ}.  Consider the 4-angulation of a hexagon below. The associated frieze pattern still has width 3 but it now has values in $\mathbb{Z}[\sqrt{2}]$ since $\lambda_4 = \sqrt{2}$.
 
 \begin{center}
  \raisebox{-.4\totalheight}
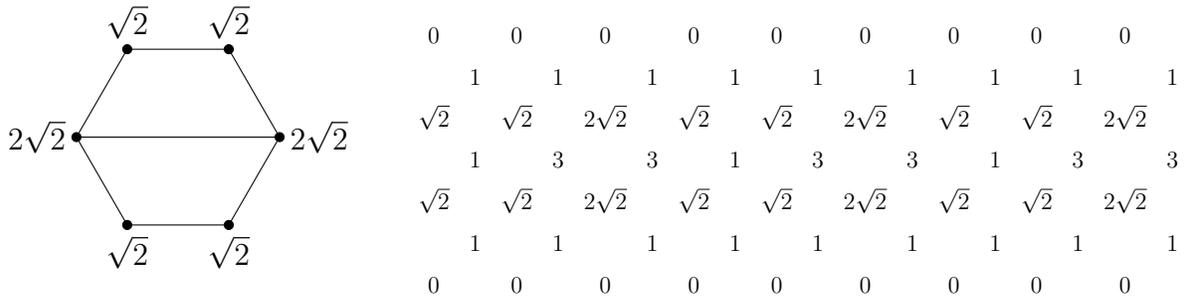
{\begin{tikzpicture}[scale=0.5]
    \newdimen\R
    \R=2.7cm
    \draw(0:\R) node[right]{$2\sqrt{2}$} -- (60:\R)node[above]{$\sqrt{2}$} -- (120:\R)  node[above]{$\sqrt{2}$}-- (180:\R) node[left]{$2\sqrt{2}$} -- (240:\R) node[below]{$\sqrt{2}$}-- (300:\R) node[below]{$\sqrt{2}$}-- (360:\R);
    \draw(180:\R) -- (0:\R);
\draw[fill=black] (0:\R) circle (.12cm);
\draw[fill=black] (60:\R) circle (.12cm);
\draw[fill=black] (120:\R) circle (.12cm);
\draw[fill=black] (180:\R) circle (.12cm);
\draw[fill=black] (240:\R) circle (.12cm);
\draw[fill=black] (300:\R) circle (.12cm);
\end{tikzpicture}}
 \scalebox{0.75}{$\begin{array}{cccccccccccccccccccc}
   &&0&&  0&&  0&&  0&&  0&&  0&&  0&&  0&&0&\\
  &&&  1&&  1&&  1&&  1&&  1&&  1&&  1&&1&&1\\
 &&\sqrt{2}&&\sqrt{2}&&2\sqrt{2}&&  \sqrt{2}&&  \sqrt{2}&&  2\sqrt{2}&&  \sqrt{2}&&\sqrt{2}&&2\sqrt{2}&\\
  &&&1&&  3&&  3&& 1&&  3&&  3&&1&&3&&3\\
 &&\sqrt{2}&&\sqrt{2}&&2\sqrt{2}&&  \sqrt{2}&&  \sqrt{2}&&  2\sqrt{2}&&  \sqrt{2}&&\sqrt{2}&&2\sqrt{2}&\\
   &&&  1&&  1&&  1&&  1&&  1&&  1&&  1&&1&&1\\
     &&0&&  0&&  0&&  0&&  0&&  0&&  0&&  0&&0&\\
 \end{array}$}
 \end{center}
 
 Finally,  we consider a dissection of a hexagon which consists of subgons of different sizes.
 
 \begin{center}
  \raisebox{-.4\totalheight}{\begin{tikzpicture}[scale=0.5]
    \newdimen\R
    \R=2.7cm
    \draw(0:\R) node[right]{$2+\sqrt{2}$} -- (60:\R)node[above]{$1$} -- (120:\R)  node[above]{$2$}-- (180:\R) node[left]{$1+\sqrt{2}$} -- (240:\R) node[below]{$\sqrt{2}$}-- (300:\R) node[below]{$\sqrt{2}$}-- (360:\R);
    \draw(180:\R) -- (0:\R);
    \draw(0:\R) -- (120:\R);
    \draw[fill=black] (0:\R) circle (.12cm);
\draw[fill=black] (60:\R) circle (.12cm);
\draw[fill=black] (120:\R) circle (.12cm);
\draw[fill=black] (180:\R) circle (.12cm);
\draw[fill=black] (240:\R) circle (.12cm);
\draw[fill=black] (300:\R) circle (.12cm);
\end{tikzpicture}}
 \scalebox{0.75}{$\begin{array}{cccccccccccccccccc}
   &&0&&  0&&  0&&  0&&  0&&  0&&  0&&  0&\\
  &&&  1&&  1&&  1&&  1&&  1&&  1&&  1&&1\\
 &&2&&1&&2+\sqrt{2}&&  \sqrt{2}&&  \sqrt{2}&&  1+\sqrt{2}&&  2&&1&\\
  &&&1&&  1+\sqrt{2}&&  1+2\sqrt{2}&& 1&&  1+\sqrt{2}&&  1+2\sqrt{2}&&1&&1+\sqrt{2}\\
 &&\sqrt{2}&&\sqrt{2}&&1+\sqrt{2}&&  2&& 1&&  2+\sqrt{2}&&  \sqrt{2}&&\sqrt{2}&\\
   &&&  1&&  1&&  1&&  1&&  1&&  1&&  1&&1\\
     &&0&&  0&&  0&&  0&&  0&&  0&&  0&&  0&\\
 \end{array}$}
 \end{center}
\end{example}

\Cref{thm:CCFrieze} and \Cref{thm:HJ} have immediate,  enumerative corollaries.

\begin{corollary}
\label{cor:CountingFPs}
Frieze patterns of  type $\Lambda_{m+2}$ and width $n$ only exist if $n = m\ell + 2$.  
If  $n = m\ell + 2$, then there are $c_\ell^{(m)}$ frieze patterns of type $\Lambda_{m+2}$ and width $n$. 
In particular, there are $c_{n-2}$ positive integral frieze patterns of width $n$.
\end{corollary}

To discuss a geometric model for infinite frieze patterns, we will need to extend \Cref{def:FriezePatternsFromDissections}. Let $S$ be an arbitrary orientable surface with nonempty boundary and a finite set of marked points.  
Assume moreover that every boundary component of $S$ contains at least one marked point.  
A dissection $T$ of $S$ is said to be \emph{cellular} if all connected components in $S \backslash T$ are discs.

\begin{definition}\label{defn:FriezePatternGeneralSurface}
Let $S$ be a surface with nonempty boundary and let $T$ be a cellular dissection of $S$.  
Fix one boundary component of $S$ and label the vertices of this boundary component $v_0,\ldots,v_{n-1}$ in clockwise order.  
For each $0 \leq i \leq n-1$,  let $h_i$, be the intersection of a small circle centered at $v_i$ with $S$.  
Let $\mu_p(i)$ denote the number of $p$-gons incident to $v_i$, where subgons intersected in multiple places by $h_i$ contribute multiple times.  
Let $F_T$ be the frieze pattern with quiddity row given by setting $F_T(i-1,i+1) = \sum_{p \geq 3}\mu_p(\overline{i})\lambda_p$.
\end{definition}

The following example highlights the use of $h_i$ in the above definition.

\begin{example}\label{ex:InfiniteFriezePatterns}
We apply \Cref{defn:FriezePatternGeneralSurface} to a triangulation of $P_6^\bullet$.  
Notice that even though globally we see that the vertex labeled $7$ is incident to 6 triangles, the ``folded'' triangle contributes 2 to the sum. 
The resulting frieze pattern is infinite; we display the first few rows.

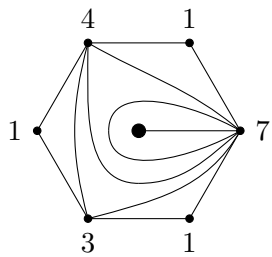
\begin{center}
\raisebox{-.4\totalheight}{\begin{tikzpicture}[scale=0.5]
   \newdimen\R
    \R=2.7cm
    
    \draw (0:\R) \foreach \x in {60,120,...,360} {  -- (\x:\R) };   

    \foreach \x/\l/\p in
     {60/{\small $1$}/above,
      120/{\small $4$}/above,
      180/{\small $1$}/left,
      240/{\small $3$}/below,
      300/{\small $1$}/below,
     360/{\small $7$}/right
     }
    \node[inner sep=1pt,circle,draw,fill,label={\p:\l}] at (\x:\R) {};

    \draw[fill=black] (0,0) circle (.18cm);
    
    \draw(\R,0) to [out = 150, in = 90] (-0.8,0);
    \draw(\R,0) to [out = 210, in = -90] (-0.8,0);
    \draw(\R,0) -- (0,0);
    \draw(\R,0) to [out = 140, in = -30] (2.7*-0.5,2.7*0.866);
    \draw(\R,0) to [out = 220, in = 0] (0,-1.4);
    \draw(0,-1.4) to [out = 180, in = -90] (2.7*-0.5,2.7*0.866);
    \draw(\R,0) to [out = 240, in = 15] (2.7*-0.5,-2.7*0.866);
    \draw (2.7*-0.5,-2.7*0.866) to [out = 105, in = 255] (2.7*-0.5,2.7*0.866);
\end{tikzpicture}}
 \scalebox{0.75}{$\begin{array}{cccccccccccccccccc}
   &&0&&  0&&  0&&  0&&  0&&  0&&  0&&  0&\\
  &&&  1&&  1&&  1&&  1&&  1&&  1&&  1&&1\\
 &&1&&7&&1&&  3&& 1&& 4&&  1&&7&\\
  &&&6&&  6&&  2&& 2&&  3&&  3&&6&&6\\
 &&17&&5&&11&&  1&& 5&&  2&&17&&5&\\
   &&&  14&&  9&&  5&&  2&&  3&&  11&&  14&&9\\
  &&9&&5&&4&&  9&& 1&&  16&&9&&5&\\
    &&&  16&&  11&&  7&&  4&&  7&&  13&&  16&&11\\
 \end{array}$}
 \end{center}
 
We also illustrate the first few rows of an infinite frieze pattern from a dissection of an annulus.
 
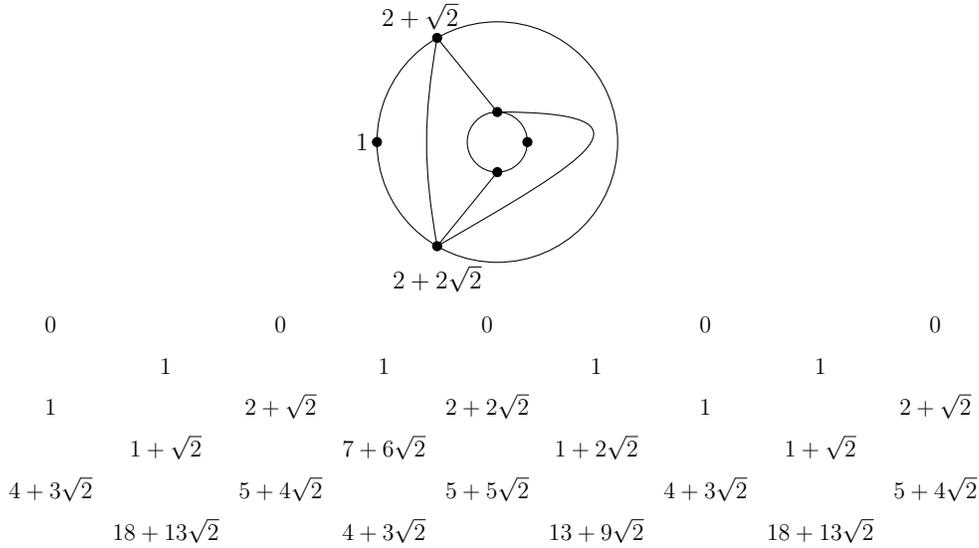
\begin{center}
  \raisebox{-.4\totalheight}{\begin{tikzpicture}[scale = 1]
 \draw (0,0) circle (1.6cm);
\draw (0,0) circle (.4cm);
 \coordinate(A) at (120:1.6cm);
 \node[above, scale = 0.8, xshift = -8pt] at (120:1.6cm){$2+\sqrt{2}$};
 \node[left, scale = 0.8] at (180:1.6cm){$1$};
  \node[below, scale = 0.8, yshift = -5pt] at (240:1.6cm){$2+2\sqrt{2}$};
 \coordinate(B) at (240:1.6cm);
 \coordinate(C) at (90:.4cm);
 \coordinate(D) at (270:.4cm);
 \draw[fill=black] (90:.4cm) circle (.06cm);
\draw[fill=black] (270:.4cm) circle (.06cm);
\draw[fill=black] (240:1.6cm) circle (.06cm);
 \draw[fill=black] (120:1.6cm) circle (.06cm);
  \draw[fill=black] (180:1.6cm) circle (.06cm);
  \draw[fill=black] (0:0.4cm) circle (.06cm);
\draw (B) to [out = 100, in = -100] (A);
\draw(A)--(C);
\draw(B) -- (D);
\draw(B) to [out = 30, in = 0, looseness = 3] (C);
  \end{tikzpicture}}
 \scalebox{0.75}{
$ \begin{array}{cccccccccccccccccccc}
  0&&0&&0&&0&&0&&\\
 &1& &1&&1&&1&\\
 1&&2+\sqrt{2}&&2+2\sqrt{2}&&1&&2+\sqrt{2}&&\\
&1+\sqrt{2}& &7 + 6\sqrt{2}&&1+2\sqrt{2}&&1+\sqrt{2}&\\
  4+3\sqrt{2}&&5+4\sqrt{2}&&5+5\sqrt{2}&&4+3\sqrt{2}&&5+4\sqrt{2}&&\\
&18+13\sqrt{2}& &4 + 3\sqrt{2}&&13+9\sqrt{2}&&18+13\sqrt{2}&\\
 \end{array}$}
 \end{center}
\end{example}

Baur,  Fellner,  Parsons, and Tschabold studied the growth behavior of periodic,  infinite frieze patterns.  This growth is encoded in a family of integers, called \emph{growth coefficients}, which measure how fast diagonals in the frieze pattern grow.  
Let $T_k(x)$ denote the $k$th normalized Chebyshev polynomial of the first kind, defined by the recurrence 
\[
T_k(x) = xT_{k-1}(x) - T_{k-2}(x) \text{ for } k \geq 2 \qquad  T_0(x) = 2 \quad T_1(x) = x.
\]

\begin{theorem}
[\cite{baur2019growth}]\label{thm:GrCo}
Let $F$ be an infinite frieze pattern and let $j$ be the minimal period of the quiddity row.  Then,  for all $k \geq 1$,  $F(i,i+j k + 1) - F(i+1,i+jk)$ is constant for all $i \in \mathbb{Z}$. 
Moreover, if $s_k$ denotes this constant difference,  then $s_k = T_k(s_1)$ for all $\ell \geq 1$.
\end{theorem}

We will primarily focus on the first growth coefficient, $s_1$, since determines all others. Sometimes $s_1$ is also called the \emph{principal growth coefficient}. 
For example, using the minimal periods,  we have that the principal growth coefficients in the first example is 2 and in the second example is $3 + 3\sqrt{2}$. 

Baur,  Parsons, and Tschabold gave a geometric model for all infinite, positive integral frieze patterns.  The following is a combination of this characterization and a characterization of the growth coefficients.

\begin{theorem}[\cite{baurInfiniteFriezes}]
\label{thm:CharacterizeInfiniteFriezePatterns}
All infinite, positive integral frieze patterns arise from a triangulation of $P_n^\bullet$ or of an annulus.  A frieze pattern arises from a triangulation of $P_n^\bullet$ if and only if its principal growth coefficient is 2.
\end{theorem}

Using the defining recurrence, we can see $T_k(2) = 2$ for all $k \geq 1$ and that either all growth coefficients of a frieze pattern are at least 2 or the frieze is not simultaneously positive, integral and infinite. Therefore,  one can replace the last condition of the Theorem with the condition that \emph{any} growth coefficient is 2.  

The second author and Chen investigated infinite frieze patterns from dissections of $P_n^\bullet$ and annuli.  They gave a description of frieze patterns from general dissections via a ``realizability algorithm'' on the quiddity row. 
If the frieze pattern is of type $\Lambda_{m+2}$, the result simplifies to an $(m+2)$-angulated version of \Cref{thm:CharacterizeInfiniteFriezePatterns}.

\begin{theorem}[\cite{banaian2021periodic}]
\label{thm:CharacterizeInfiniteFriezePatternsMAngulation}
All infinite, positive frieze patterns of type $\Lambda_{m+2}$ arise from an $(m+2)$-angulation of $P_n^\bullet$ or of an annulus. 
Moreover, an infinite frieze pattern arises from an $(m+2)$-angulation of $P_n^\bullet$ if and only if its principal growth coefficient is 2.
\end{theorem}

We conclude this section with a combinatorial interpretation of entries of a frieze pattern from a dissection which will be useful in future proofs. Given a polygon $P_n$ with dissection $T$, we say an \emph{(admissible) matching} between $v_i$ and $v_j$ is a sequence $Q_{i+1},\ldots,Q_{j-1}$ such that each $Q_k$ is a subgon incident to $v_k$ and such that each $p$-gon from the dissection appears on the list at most $p-2$ times. 

These matchings were introduced in the case of triangulations by Broline, Crowe, and Isaacs \cite{BCI}. They were extended to dissections by the second author and Chen in \cite{banaian2021periodic}, inspired by a construction by Bessenrodt \cite{Bessenrodt}. 
For general dissections, there is a weight, denoted $wt_T$, associated to each matching based on the number of times every subgon appears on the list. 
We do not need the details here, beyond the fact that $wt_T(Q_{i+1},\ldots,Q_{j-1}) \geq 1$ and in the triangulation case, $wt_T(Q_{i+1},\ldots,Q_{j-1}) = 1$ for all admissible sequences, recovering the work of \cite{BCI}. 

These constructions can be extended to frieze patterns from punctured polygons by lifting of $P_n^\bullet$ with dissection $T$ to a periodic dissection of an infinite strip. 
This was done in the triangulation case by Baur, Parsons, and Tschabold \cite{baurInfiniteFriezes} and in the dissection case in \cite{banaian2021periodic}.  
In this process, the vertex $v_i$ lifts to a family of vertices $\{\overline{v_{n\ell + i}}\}_{\ell \in \mathbb{Z}}$.
An arc of the form $(v_i,v_j)$ lifts to the family of arcs between lifts of $v_i$ and $v_j$. Meanwhile,
spokes lift to an infinite family of \emph{asymptotic arcs}, which have one endpoint and travel infinitely far in one direction in the disc. An example is given in \Cref{ex:BCI}.

We summarize all of these results in the following, focusing only on the cases we will use. 

\begin{theorem}
[\cite{banaian2021periodic,baurInfiniteFriezes,BCI}]\label{thm:CombinatorialInterpretation}~
\begin{enumerate}
    \item If $F_T$ is the finite frieze pattern from a dissection $T$ of a polygon $P_n$, then $F_T(i,j)$ is the weighted sum, under $wt_T$, of admissible matchings between $v_{\overline{i}}$ and $v_{\overline{j}}$. 
    \item If $F_T$ is the infinite frieze pattern from a dissection $T$ of a once-punctured polygon $P_n^\bullet$, then  $F_T(i,j)$ is the weighted sum, under $wt_T$, of admissible matchings between $\overline{v_i}$ and $\overline{v_j}$ in the lift of $(P_n^\bullet, T)$ to the infinite strip. 
\end{enumerate}
\end{theorem}

\begin{example}\label{ex:BCI}
Here, we illustrate both parts of \Cref{thm:CombinatorialInterpretation}. Below, we reproduce the triangulation from \Cref{ex:BijectionsToFPs}, with vertices and subgons labeled. Using this indexing, we have $F(1,4) = 5$, and 5 is also the cardinality of the admissible matchings between $v_1$ and $v_4$. These are $\alpha,\mu; \alpha,\delta; \beta,\mu; \beta,\delta;$ and  $\mu,\delta$.

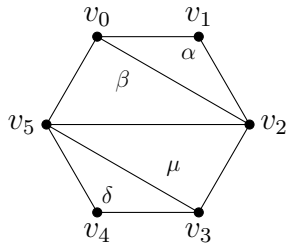
\begin{center}
\begin{tikzpicture}[scale=0.5]
    \newdimen\R
    \R=2.7cm
    \draw(0:\R) node[right]{$v_2$} -- (60:\R)node[above]{$v_1$} -- (120:\R)  node[above]{$v_0$}-- (180:\R) node[left]{$v_5$} -- (240:\R) node[below]{$v_4$}-- (300:\R) node[below]{$v_3$}-- (360:\R);
    \draw(-60:\R) -- (180:\R) -- (0:\R) -- (120:\R);
    \node[scale=0.75] at (120:.5\R){$\beta$};
    \node[scale=0.75] at (60:.8\R){$\alpha$};
    \node[scale=0.75] at (300:.5\R){$\mu$};
    \node[scale=0.75] at (240:.8\R){$\delta$};
    \draw[fill=black] (0:\R) circle (.12cm);
\draw[fill=black] (60:\R) circle (.12cm);
\draw[fill=black] (120:\R) circle (.12cm);
\draw[fill=black] (180:\R) circle (.12cm);
\draw[fill=black] (240:\R) circle (.12cm);
\draw[fill=black] (300:\R) circle (.12cm);
\end{tikzpicture}
\end{center}

Next, we consider the triangulation $T$ of $P^\bullet_6$ in \Cref{ex:InfiniteFriezePatterns}. Below, we draw its lift to the universal strip and we again label the subgons. The subgons labeled $\delta_i$ have one endpoint at $+\infty$. With this indexing, $F_T(3,9) = 5$, which is the cardinality of admissible matchings between $\overline{v_4}$ and $\overline{v_8}$. These are $\eta_0,\mu_0,\alpha_1; \eta_0,\delta_0,\alpha_1; \eta_0,\delta_1,\alpha_1; \eta_0,\mu_1,\alpha_1; $ and $\eta_0,\beta_1,\alpha_1$. For another example, we have $F_T(0,4)=1$, and there is a unique admissible matching between $\overline{v_0}$ and $\overline{v_4}$, namely, $\alpha_0,\beta_0,\epsilon_0$.

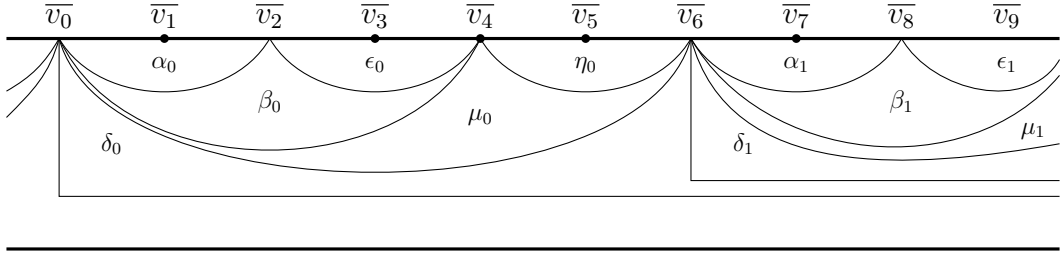
\begin{center}
\begin{tikzpicture}[scale=0.7]
\draw[very thick] (0,0) -- (20,0);
\draw[very thick] (0,4) -- (20,4);
\node[above] at (1,4){$\overline{v_0}$};
\node[above] at (3,4){$\overline{v_1}$};
\node[circle, fill = black, scale = 0.3] at (3,4){};
\node[above] at (5,4){$\overline{v_2}$};
\node[above] at (7,4){$\overline{v_3}$};
\node[circle, fill = black, scale = 0.3] at (7,4){};
\node[above] at (9,4){$\overline{v_4}$};
\node[above] at (11,4){$\overline{v_5}$};
\node[circle, fill = black, scale = 0.3] at (11,4){};
\node[above] at (13,4){$\overline{v_6}$};
\node[above] at (15,4){$\overline{v_7}$};
\node[circle, fill = black, scale = 0.3] at (15,4){};
\node[above] at (17,4){$\overline{v_8}$};
\node[above] at (19,4){$\overline{v_9}$};
\node[circle, fill = black, scale = 0.3] at (9,4){};
\draw(1,4) -- (1,1) -- (20,1);
\draw(13,4) -- (13,1.3) -- (20,1.3);
\draw(1,4) to [out = -75, in = 270-15, looseness = 0.75] (13,4);
\draw(5,4) to [out = -60, in = 270-30] (9,4);
\draw(9,4) to [out = -60, in = 270-30] (13,4);
\draw(13,4) to [out = -60, in = 270-30] (17,4);
\draw(1,4) to [out = -60, in = 270-30] (5,4);
\draw(1,4) to [out = -65, in = 270-25] (9,4);
\draw(13,4) to [out = -60, in = 230] (20,3.3);
\draw(13,4) to [out = -75, in = 190] (20,2);
\draw(17,4) to [out = -60, in = 240] (20,3.6);
\draw(0,3) to [out = 30, in = 240](1,4);
\draw(0,2.5) to [out = 45, in = 255](1,4);
\node[scale=0.8] at (3,3.5){$\alpha_0$};
\node[scale=0.8] at (5,2.8){$\beta_0$};
\node[scale=0.8] at (9,2.5){$\mu_0$};
\node[scale=0.8] at (2,2){$\delta_0$};
\node[scale=0.8] at (7,3.5){$\epsilon_0$};
\node[scale=0.8] at (11,3.5){$\eta_0$};
\node[scale=0.8] at (15,3.5){$\alpha_1$};
\node[scale=0.8] at (17,2.8){$\beta_1$};
\node[scale=0.8] at (19.5,2.25){$\mu_1$};
\node[scale=0.8] at (19,3.5){$\epsilon_1$};
\node[scale=0.8] at (14,2){$\delta_1$};
\end{tikzpicture}
\end{center}
\end{example}

\subsection{Counting Frieze Patterns up to Symmetry}
\label{subsec: counting frieze patterns up to symmetry}

Our goal in this section is to enumerate frieze patterns in several ways. In particular, we will enumerate finite frieze patterns up to global row shift, and we will enumerate certain infinite frieze patterns. This will use an alternate description of the cyclic sieving phenomenon. The equivalence of the following and \Cref{def:CSP} is shown in \cite{reiner2004cyclic}.

\begin{proposition}[\cite{reiner2004cyclic}]\label{prop:CSPVersion2}
A triple $(X,G_n = \langle \sigma \rangle,f(q))$ exhibits the cyclic-sieving phenomenon if and only if each of the coefficients $g_i$ determined by \[
X(q) \equiv \sum_{i=0}^{n-1} g_i q^i \mod{q^n - 1}
\]
are equal to the number of $G$-orbits on $X$ for which the stabilizer-order divides $i$.
\end{proposition}

Now, once we show that the operations of rotating a dissection $T$ and shifting all rows in $F_T$ agree, we can use \Cref{prop:CSPVersion2} to enumerate equivalence classes of frieze patterns. In the case of frieze patterns of type $\Lambda_{m+2}$, this largely follows from combining existing literature.


\begin{lemma}
\label{lem:ShiftAndRotate}
If $T$ and $T'$ are two dissections of $P_n$ or $P_n^\bullet$ where $T'$ is the result of rotating $T$ by $2\pi/n$ in the counterclockwise direction, then for all $0 \leq j-i \leq n$, $F_{T'}(i,j) = F_T(i+1,j+1)$.
\end{lemma}

\begin{proof}
The statement is trivial for the first two rows, and by construction of $F_T$ and $F_{T'}$ it is clear for the quiddity row. Now, if $j-i > 2$, this can be shown with \Cref{lem:Kontinuant}.
\end{proof}

Conway and Coxeter showed that the rows in a frieze pattern of width $n$ must be $n$-periodic. Hence, global row shift is a cyclic action on the set of frieze patterns of fixed width, and one can easily export any cyclic sieving phenomenon from dissections of $P_n$ to frieze patterns of width $n$. 

We will now turn our attention to utilizing the cyclic sieving phenomenon to enumerate equivalence classes of frieze patterns.
In the following, if $n = m\ell + 2$, let $\widetilde{c}_\ell^{(m)}(q)$ denote the degree $n-1$ polynomial which is equivalent to $c_\ell^{(m)}(q)$ modulo $q^n - 1$.

\begin{proposition}
\label{thm:num of equiv classes of friezes of type lambda}
Let $m$ be a positive integer and let $n = m\ell + 2$.  
The number of equivalence classes of positive frieze patterns of type $\Lambda_{m+2}$ and of width $n$ up to global row shift is the constant term of $\widetilde{c}_\ell^{(m)}(q)$.  
\end{proposition}

\begin{proof}
First, let $m = 1$.  By \cite[Theorem 7.1]{reiner2004cyclic}, the triple $(\mathcal{T}_n, G_n, c_{n-2}(q))$ exhibits the cyclic sieving phenomenon where $\mathcal{T}_n$ is the set of triangulations of $P_n$. By \Cref{prop:CSPVersion2}, this means that the constant term of $\widetilde{c}^{(1)}_{n-2}(q) \pmod{q^n - 1}$ is equal to the total number of cyclic equivalence classes of triangulations of $P_n$. Then,  \Cref{thm:CCFrieze} and \Cref{lem:ShiftAndRotate} imply that these equivalence classes are in bijection with equivalence classes of positive, integral frieze patterns of width $n$ up to global row shift.

If $m > 1$, the same logic holds, applying instead the cyclic sieving result in \cite[Theorem 3.8]{eu2008cyclic} and the bijection in \Cref{thm:HJ}. 
\end{proof}

\begin{remark}
A triangulation $T$ of $P_n$ determines a cluster in a type $A_{n-3}$ cluster algebra, $\mathcal{A}$, as well as a cluster-tilting object $S_T$ in the associated cluster category, $\mathcal{C}_{\mathcal{A}}$ \cite{AnCategorication}. One can go from the  $\mathcal{C}_{\mathcal{A}}$ to $\mathcal{A}$ using a \emph{cluster character map}, as was first defined in \cite{caldero2006cluster}. Palu described how to build a cluster character map for any given cluster-tilting object \cite{Palu}. 

It is common to draw the \emph{Auslander-Reiten (AR) quiver} of $\mathcal{C}_{\mathcal{A}}$ which has vertices labeled by indecomposable objects and which has the same overall shape as the frieze pattern $F_T$. Indeed, one can view $F_T$ as the result of applying the cluster character map associated to $S_T$ to each vertex/object in the AR quiver, yielding cluster variables (in particular, Laurent polynomials), then further specializing the initial variables in each expression to 1, yielding integers. 

Rotation of $P_n$ also has a meaning in this setting; it coincides with \emph{Auslander-Reiten (AR) translation}, an important and prominent functor on $\mathcal{C}_{\mathcal{A}}$. Therefore, studying equivalence classes of positive integral, finite frieze patterns under rotation/global row shift is equivalent to studying equivalence classes of cluster character maps associated to cluster tilting objects under AR translation.

This remark does not hold for the positive integral, infinite frieze patterns considered here as our action does not coincide with AR translation anymore. See \Cref{prop:CSPPuncturedPolygonTriangulationsOnly}.
\end{remark}

Next,  fix a positive integer $n$ and a finite set $\mathcal{M}= \{m_1,\ldots,m_s\} \subset \mathbb{Z}_{> 0}$.  
Let $a_{\mathcal{M}}(q) := \sum_\mu a_\mu(q)$ where we sum over all $\mu = (\mu_1,\ldots,\mu_n) \in \mathbb{Z}^n$ such that $\sum i \mu_i = n-2$ and $\mu_i > 0$ only if $i \in \mathcal{M}$. 
By construction, $a_{\mathcal{M}}(1)$ is the number of dissections of an $n$-gon which only use subgons with sizes from  $\mathcal{M}$. Let $\widetilde{a_{\mathcal{M}}}(q)$ denote the degree $n-1$ polynomial which is equivalent to $a_{\mathcal{M}}(q)$ modulo $q^n - 1$.

\begin{proposition}
Fix a positive integer $n$ and a finite set $\mathcal{M}= \{m_1,\ldots,m_s\} \subset \mathbb{Z}_{>0}$. 
The number of equivalence classes of frieze patterns from dissections of width $n$ and with values in $\mathbb{Z}[\lambda_{m_1+2},\ldots,\lambda_{m_s+2}]$ up to global shift of rows is the constant term of $\widetilde{a_{\mathcal{M}}}(q)$.
\end{proposition}

\begin{proof}
This result follows from \Cref{thm:CSPAmu}, \Cref{prop:CSPVersion2}, and \Cref{lem:ShiftAndRotate}. 
\end{proof}

Next, we use our cyclic sieving result on dissections of $P_n^\bullet$ to count infinite frieze patterns of type $\Lambda_{m+2}$ with a fixed minimal period of their rows. 
Neither \Cref{thm:CharacterizeInfiniteFriezePatterns} nor \Cref{thm:CharacterizeInfiniteFriezePatternsMAngulation} is phrased as a bijection. 
This is because, given a dissection $T$ of $P_n^\bullet$ or an annulus, we can construct a \emph{$k$-fold} dilation of $T$ for any $k \geq 1$ which would yield the same frieze pattern.  
The two dissections in \Cref{fig:SymmetricMangulation} are examples of this process.  
This dilation will not change the type of underlying surface, and a consequence of the constructive proofs in \cite{banaian2021periodic, baurInfiniteFriezes} is that there is a dissection associated to a minimal period of the quiddity row. 

Recall we have assumed that all infinite frieze patterns are periodic. A consequence of \Cref{lem:Kontinuant} is that the period of any row will divide the period of the quiddity row. Therefore, the period of the quiddity row is the \emph{minimal period} of all rows of an infinite frieze pattern.

We begin by considering the integral case,  for which we have a complete result.  Let $\widetilde{t_n}(q)$ denote the unique polynomial of degree $n-1$ which is equivalent to $t_n(q)$ modulo $q^n-1$. 

\begin{theorem}\label{thm:CountInfiniteIntegralFrieze}
The number of infinite, positive integral frieze patterns with principal growth coefficient 2 and minimal period $n$, counted up to global shift of the rows, is $[q^1] \widetilde{t_n}(q)$. 
Equivalently, the number of infinite, positive integral frieze patterns with rows with principal growth coefficient 2 and minimal period $n$ is $n[q^1] \widetilde{t_n}(q)$.
\end{theorem}

\begin{proof}
An infinite, positive integral frieze pattern with principal growth coefficient 2 with whose rows are $n$-periodic arises from a triangulation $T$ of $P^\bullet_n$ by \Cref{thm:CharacterizeInfiniteFriezePatterns}. If $n$ is the minimal period of the rows, then $T$ is not a $k$-fold dilation of a triangulation $T'$ of $P^\bullet_{n'}$ with $n' < n$.  This is equivalent to saying that $T$ does not have nontrivial rotational symmetry. Therefore, the stabilizer of $T$ in $G_n = \langle \sigma\rangle$, the cyclic group acting on triangulations of $P_n^\bullet$ by rotation, is the identity element. Combining \Cref{prop:CSPPuncturedPolygonTriangulationsOnly} and \Cref{prop:CSPVersion2} shows that the number of such orbits is the linear term of $\widetilde{t_n}(q)$. 

If we want to count all such frieze patterns instead of grouping them by shift equivalence classes, we can simply multiply this number by $n$.  The absence of symmetry in these dissections ensures this does not overcount.  


\end{proof}

The sequence $[q^1]\widetilde{t_n}(q)$ begins $1,1,3,8,27,245,800,\ldots$ and appears to coincide with \href{https://oeis.org/A022553}{OEIS A022553}. The sequence $n[q^1]\widetilde{t_n}(q)$ also seems to appear as \href{https://oeis.org/A045630}{OEIS A045630}.

\begin{quest}
Can infinite, positive integral frieze patterns with principal growth coefficient $s > 2$ for a fixed $s$ be counted in a similar way?
\end{quest}

When considering $(m+2)$-angulations for $m > 1$, due to the restriction in \Cref{prop:CSPPuncturedPolygon}, we cannot consider all possible numbers of spokes.  
In order to convert our result on $(m+2)$-angulations of $P_n^\bullet$ to a result about frieze patterns, we explain how the number of spokes of a dissection affects the associated frieze pattern. 
In the following, note that a \emph{fundamental domain} of an infinite frieze pattern $F_T$ with $n$-periodic rows is an infinite set of the form $\{F_T(i,j): k \leq i \leq k + n -1, j \geq i\}$ for any $k \in \mathbb{Z}$. 

\begin{proposition}
\label{prop:m-angulation with s spokes to frieze}
If $F_T$ is an infinite frieze pattern of type $\Lambda_{m+2}$ associated to a dissection $T \in T_{n,s}^{\bullet,(m)}$, then there are $\frac{n}{m}-s$ entries 1 in any fundamental domain of $F_T$.
\end{proposition}

\begin{proof}
In \cite[Corollary 2]{BCI}, Broline, Crowe, and Isaacs show that, given a polygon $P_n$ with triangulation $T$, if $(v_i,v_j) \in T$, then there is only one admissible matching between $v_i$ and $v_j$. The same argument holds in the more general case of a polygon with dissection. 

Recall that when working with a frieze pattern from a dissection $T$ of $P_n^\bullet$, we first construct a lift of $T$ to the infinite strip. However, every arc in the infinite strip can be regarded as a diagonal in a finite polygon inside the strip, allowing us to use the same machinery as in the finite case. 
Since spokes lift to one-sided infinite arcs with only one endpoint, entries of 1 in an infinite frieze pattern from a dissection of $P_n^\bullet$ will correspond to lifts of non-spokes. Each non-spoke from $T$ will correspond to one entry in a fundamental domain of $F_T$. An $(m+2)$-angulation of $P_n^\bullet$ will use $\frac{n}{m}$ arcs, so there will be $\frac{n}{m} - s$  non-spokes. By our previous discussion, we can conclude there are $\frac{n}{m} - s$  entries 1 in a fundamental domain of $F_T$.
\end{proof}

Let $\widetilde{t}_{n,s}^{(m)}(q)$ denote the unique polynomial of degree $n-1$ which is equivalent to $t_{n,s}^{(m)}(q)$ modulo $q^n-1$. 

\begin{proposition}\label{prop:CountInfiniteTypeLambdapFrieze}
Let $n,m,$ and $s$ be a triple satisfying the conditions from \Cref{prop:CSPPuncturedPolygonm-angulation}. The number of infinite frieze patterns of type $\Lambda_{m+2}$ with growth coefficient 2 having minimal period $n$ and $\frac{n}{m} - s$ entries 1 in any fundamental domain, counted up to global shift of the rows, is $[q^1] \widetilde{t}_{n,s}^{(m)}(q)$. Equivalently, the total number of infinite frieze patterns of type $\Lambda_{m+2}$ with growth coefficient 2 having minimal period $n$ and $\frac{n}{m} - s$ entries 1 in any fundamental domain is $n [q^1] \widetilde{t}_{n,s}^{(m)}(q)$.
\end{proposition}

The proof of \Cref{prop:CountInfiniteTypeLambdapFrieze} is analogous to that of \Cref{thm:CountInfiniteIntegralFrieze}.

\begin{example}
 As previously noted, the triple $n = 12, m = 3, s = 3$ satisfies the conditions in \Cref{prop:CSPPuncturedPolygonm-angulation}. Therefore, the number of infinite frieze patterns of type $\Lambda_5$ with growth coefficient 2, minimal period 12, and one entry 1 in any fundamental domain, counted up to global shift, is 
 $
 [q^1] t_{12,3}^{(4)} 
 = [q^1] \big(3
 \mathop{\Bigl[ \begin{smallmatrix}
 \scriptstyle 12 \\
 \scriptstyle 11
 \end{smallmatrix} \Bigr]}_q
 \big)= 3
 $. 
 Equivalently, the total number of such frieze patterns (no longer counting up to shift) is 36. 
 We draw 5-angulation representatives of the three symmetry classes in \Cref{fig:ExampleCountingInfiniteFriezesLambda5}.

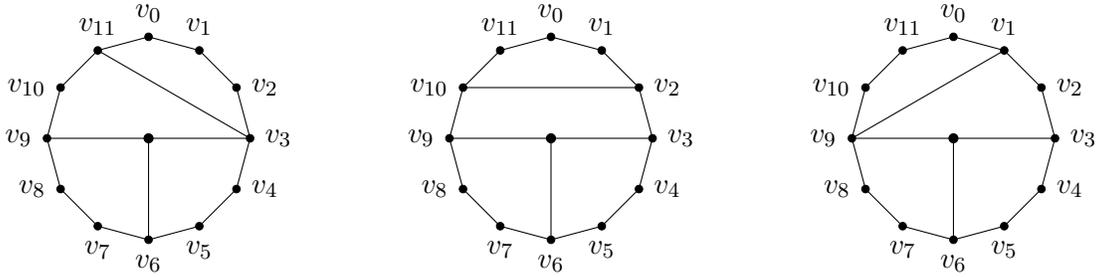
\begin{figure}
\centering
\begin{tabular}{ccc}
\begin{tikzpicture}[scale=0.5]
    \newdimen\R
    \R=2.7cm
    \draw (0:\R) \foreach \x in {30,60,...,360} {  -- (\x:\R) };   
    \foreach \x/\l/\p in
     {0/{\small $v_3$}/right,
      30/{\small $v_2$}/right,
      60/{\small $v_1$}/above,
      90/{\small $v_0$}/above,
      120/{\small $v_{11}$}/above,
      150/{\small $v_{10}$}/left,
      180/{\small $v_9$}/left,
     210/{\small $v_8$}/left,
      240/{\small $v_7$}/below,
      270/{\small $v_6$}/below,
      300/{\small $v_5$}/below,
      330/{\small $v_4$}/right}
    \node[inner sep=1pt,circle,draw,fill,label={\p:\l}] at (\x:\R) {};
    \draw[fill=black] (0,0) circle (.12cm);
    \draw(0:\R) -- (0,0);
    \draw(270:\R) -- (0,0);
    \draw(180:\R) -- (0,0);
    \draw(0:\R) -- (120:\R);
\end{tikzpicture}&
\begin{tikzpicture}[scale=0.5]
    \newdimen\R
    \R=2.7cm
    \draw (0:\R) \foreach \x in {30,60,...,360} {  -- (\x:\R) };   
    \foreach \x/\l/\p in
     {0/{\small $v_3$}/right,
      30/{\small $v_2$}/right,
      60/{\small $v_1$}/above,
      90/{\small $v_0$}/above,
      120/{\small $v_{11}$}/above,
      150/{\small $v_{10}$}/left,
      180/{\small $v_9$}/left,
     210/{\small $v_8$}/left,
      240/{\small $v_7$}/below,
      270/{\small $v_6$}/below,
      300/{\small $v_5$}/below,
      330/{\small $v_4$}/right}
    \node[inner sep=1pt,circle,draw,fill,label={\p:\l}] at (\x:\R) {};
    \draw[fill=black] (0,0) circle (.12cm);
    \draw(0:\R) -- (0,0);
    \draw(270:\R) -- (0,0);
    \draw(180:\R) -- (0,0);
    \draw(30:\R) -- (150:\R);
\end{tikzpicture}&
\begin{tikzpicture}[scale=0.5]
    \newdimen\R
    \R=2.7cm
    \draw (0:\R) \foreach \x in {30,60,...,360} {  -- (\x:\R) };   
    \foreach \x/\l/\p in
     {0/{\small $v_3$}/right,
      30/{\small $v_2$}/right,
      60/{\small $v_1$}/above,
      90/{\small $v_0$}/above,
      120/{\small $v_{11}$}/above,
      150/{\small $v_{10}$}/left,
      180/{\small $v_9$}/left,
     210/{\small $v_8$}/left,
      240/{\small $v_7$}/below,
      270/{\small $v_6$}/below,
      300/{\small $v_5$}/below,
      330/{\small $v_4$}/right}
    \node[inner sep=1pt,circle,draw,fill,label={\p:\l}] at (\x:\R) {};
    \draw[fill=black] (0,0) circle (.12cm);
    \draw(0:\R) -- (0,0);
    \draw(270:\R) -- (0,0);
    \draw(180:\R) -- (0,0);
    \draw(60:\R) -- (180:\R);
\end{tikzpicture}\\
\end{tabular}
\caption{Representatives for the three symmetry classes of infinite frieze patterns of type $\Lambda_5$ with one 1 in any fundamental domain and minimal period 12. }\label{fig:ExampleCountingInfiniteFriezesLambda5}
\end{figure}

\end{example}

\subsection{Frieze patterns from orbifolds}
\label{subsec: frieze patterns from orbifolds}

In this section, we discuss an interpretation of frieze patterns from triangulations with nontrivial symmetry. 
A triangulated polygon either has no rotational symmetry, 2-fold rotational symmetry, or 3-fold rotational symmetry. 
The corresponding frieze patterns correspond with three of the four types of friezes with glide symmetry; see \cite[Questions (3)-(5)]{conway1973triangulated,conway1973answers}. 

These three cases can also be explained by extending the notion of growth coefficients to finite frieze patterns. By insisting that a frieze pattern is tame (recall \Cref{def:BasicFriezeTerms}) and allowing negative entries, every finite frieze pattern can be uniquely extended to an infinite frieze pattern. In particular, in such an extension, the row beneath the bottom row of 0's would consist of $-1$'s.

The following was suggested in \cite[Proposition 5.9]{baur2019growth}. 
We fill in some details to prove this in general. We will heavily use item (1) from \Cref{thm:CombinatorialInterpretation}.

\begin{proposition}
\label{prop:SortFriezeByGrCo}
Let $F$ be a finite, positive integral frieze pattern of width $n$ and let $T$ be the corresponding triangulation of $P_n$; that is, $F = F_T$. 
\begin{enumerate}
    \item The principal growth coefficient of $F$ is $-2$ if and only if $T$ has no nontrivial rotational symmetry.
    \item The principal growth coefficient of $F$ is 0 if and only if $T$ has 2-fold rotational symmetry.
    \item The principal growth coefficient of $F$ is $1$ if and only if $T$ has 3-fold  rotational symmetry.
\end{enumerate}
\end{proposition}

\begin{proof}

Since every triangulation with rotational symmetry has 2-fold or 3-fold symmetry and every frieze pattern corresponds uniquely to a triangulation, it will suffice to show the three ``if'' statements hold. 

\textit{Item (1)} Saying $T$ has no nontrivial rotational symmetry is equivalent to saying that the quiddity row of $F_T$ has minimal period $n$.
In the extension of $F_T$ to an infinite frieze pattern, $F(i,i+n+1) = -1$ for all $i$, and we see the growth coefficient is $F(i,i+n+1)-F(i,i+n-1)= -1 - 1 = -2$.

\textit{Item (2)} Saying $T$ has 2-fold rotational symmetry is equivalent to saying that the quiddity row of $F_T$ has minimal period $\frac{n}{2}$. 
By \Cref{thm:CombinatorialInterpretation}, $F(0,\frac{n}{2}+1)$ is equal to the number of admissible matchings between $v_0$ and $v_{\frac{n}{2}+1}$ and $F(1,\frac{n}{2})$ is equal to the number of admissible matchings between $v_1$ and $v_{\frac{n}{2}}$. 
By \cite[Corollary 1]{BCI}, the number of admissible matchings between $v_0$ and $v_{\frac{n}{2}+1}$ (i.e choosing subgons at $v_1,v_2,\ldots,v_{\frac{n}{2}}$) is equivalent to the number of admissible matchings between $v_{\frac{n}{2}+1}$ and $v_0$ (i.e choosing subgons at $v_{\frac{n}{2}+2},v_{\frac{n}{2}+3},\ldots,v_{n-1}$). 
The latter is equal to the number of  admissible matchings between $v_1$ and $v_{\frac{n}{2}}$ because $T$ has 2-fold symmetry. Therefore, $F(0,\frac{n}{2}+1) - F(1,\frac{n}{2}) = 0$, and by \Cref{thm:GrCo}, 0 is the principal growth coefficient.

\textit{Item (3)} If $T$ has 3-fold rotational symmetry, then the quiddity row of $F_T$ has minimal period $\frac{n}{3}$. 
Such a triangulation will have a central triangle, so without loss of generality, suppose this triangle consists of diagonals $(v_0,v_{\frac{n}{3}}), (v_{\frac{n}{3}},v_{\frac{2n}{3}}),$ and $(v_{\frac{2n}{3}},v_0)$. 
We again compare admissible matchings between $v_0$ and $v_{\frac{n}{3}+1}$ with admissible matchings between $v_1$ and $v_{\frac{n}{3}}$. 
We can regard the latter pair of vertices inside the smaller polygon formed by $v_0,v_1,\ldots,v_{\frac{n}{3}}$, with the diagonal $(v_0,v_{\frac{n}{3}})$ a boundary edge; call this polygon $Q$. 
Applying \cite[Corollary 1]{BCI} to the computation of $F(1,\frac{n}{3})$ inside $Q$, we have that the number of admissible matchings between $v_1$ and $v_{\frac{n}{3}}$ can be computed by counting the number of subgons in $Q$ incident to $v_0$. That is, we have shown $F(1,\frac{n}{3})= 1 + \# \{(v_0,v_i) \in T: 0 < i < \frac{n}{3}\}$. 

Now, we consider admissible matchings between $v_0$ and $v_{\frac{n}{3}+1}$. 
By \cite[Corollary 2]{BCI}, there is a unique admissible matching between $v_0$ and $v_{\frac{n}{3}}$, and this uses every subgon in $Q$. 
Therefore, an admissible matching between $v_0$ and $v_{\frac{n}{3}+1}$ is the result of appending a subgon incident to $v_{\frac{n}{3}}$ which is not in $Q$ to this sequence. 
The number of such subgons is $1 + \# \{(v_\frac{n}{3},v_i) \in T: \frac{n}{3} < i < n\} = 2 + \#\{(v_\frac{n}{3},v_i) \in T: \frac{n}{3} < i < \frac{2n}{3}\}$, where we use our assumption about the position of the central triangle. 
Since $T$ has 3-fold symmetry, we know $\#\{(v_\frac{n}{3},v_i) \in T: \frac{n}{3} < i < \frac{2n}{3}\} = \#\{(v_0,v_i) \in T: 0 < i < \frac{n}{3}\}$, so $F(0,\frac{n}{3}+1) - F(1,\frac{n}{3}) = 1$.
\end{proof}

\begin{remark}
\Cref{prop:SortFriezeByGrCo} could easily be extended by associating to a dissection $T$ the largest $\ell$ such that $T$ has $\ell$-fold symmetry. 
For $\ell > 3$, we can use the same strategy as in the proof of Item (3), utilizing the combinatorial interpretation of entries of $F_T$ from \cite{banaian2021periodic}. 
We chose to focus on the case of integral frieze patterns and triangulations since the motivation of this section comes from the theory of generalized cluster algebras, whose generators will always specialize to integers by the Laurent Phenomenon \cite{chekhov2014teichmuller}.
\end{remark}

\Cref{prop:SortFriezeByGrCo} partitions finite, positive integral frieze patterns into three sets. 
Our goal in the remainder of this section is to discuss an interpretation of frieze patterns described in items (2) and (3). 
This will require a new perspective on frieze patterns. 

In \cite[Theorem 1.5]{holm2020p}, Holm and J{\o}rgensen show that the values $F_T$ in a frieze pattern associated to a dissection $T$ of a polygon satisfy Ptolemy relations.  
That is, for all $i < j < k < \ell$, we have 
\[
F_T(i,k) F_T(j,\ell) = F_T(i,j) F_T(k,\ell) + F_T(i,\ell) F_T(j,k).
\]

Given $T$ a dissection of $P_n$, each frieze pattern $F_T$ encodes a ring homomorphism from a \emph{type $A_{n-3}$ cluster algebra} to $\mathbb{Z}[\lambda_{p_1},\ldots,\lambda_{p_s}]$ where $p_1,\ldots,p_s$ are the sizes of the subgons cut out by $T$. 
This follows from the geometric realization of these algebras \cite[Section 12.2]{CAII}.

Studying homomorphisms from more general cluster algebras to integral domains has been a topic of interest lately; see for example \cite{pants,felikson2024friezes,FontainePlamondon,germain2023friezes,GunawanSchiffler}. 
These homomorphisms are called \emph{friezes}. 
Often, the cluster algebras studied through the lens of friezes have a geometric realization, allowing a pictorial description of the homomorphism.  
In particular, friezes on $P_n$ with positive integral values are in bijection with frieze patterns of width $n$ with positive integral values. 

We will introduce a new family of friezes here which will correspond to the frieze patterns in \Cref{prop:SortFriezeByGrCo} items (2) and (3). 
These  are related to skew-symmetrizable cluster algebras and generalized cluster algebras in the sense of Chekhov-Shaprio \cite{chekhov2014teichmuller} respectively. We will define these purely geometrically and avoid relying on the language of cluster algebras.

An orbifold is a generalization of a manifold with singular points called orbifold points. 
Let $P_n^\star$ denote a polygon $P_n$ with one orbifold point $\star$ which has an order $p \geq 2$. 
Arcs on $P_n^\star$ are equivalence classes of sets of arcs on $P_{pn}$ which are fixed under $2\pi/p$ rotation. 
Triangulations on $P_n^\star$ come from maximal sets of arcs on $P_{pn}$ which are invariant under rotation.

The family of arcs generated by all rotations of $(i,i+n)$, $0 \leq i < n$, project to an arc $\gamma$ on $P_n^\star$ with $\gamma(0) = \gamma(1) = v_i$ such that $\gamma$ cuts out a monogon containing $\star$. We refer to these as \emph{pending arcs}. 
Note that some other sources will instead draw pending arcs has having an endpoint at $v_i$ and an endpoint at $\star$; we forbid such curves and insist our arcs only have endpoints in marked points, i.e. vertices.  
We use the same conventions as in the description of arcs on $P_n^\bullet$. 
Here, the notation $(v_i,v_i)$ will mean the pending arc based at $v_i$; non-pending arcs will be called \emph{standard arcs} for emphasis.

Consider a function $f$ on the set of arcs of $P_n^\star$. 
Given two arcs $\tau, \tau'$ which intersect, we say that the product $f(\tau)f(\tau')$ \emph{respects skein relations} if $f(\tau)f(\tau') = f(\Gamma^+) + f(\Gamma^-)$ where $\Gamma^+$ and $\Gamma^-$ are sets of arcs resulting from taking the chosen intersection point $\mathsf{X}$ of $\tau$ and $\tau'$ and replacing it with $\asymp$ and \rotatebox[origin=c]{90}{$\asymp$} respectively. 
When performing this smoothing on an orbifold it is possible to produce a curve with a self-intersection. 
In this case, one proceeds by smoothing the self-crossing in the same way. Smoothing a self-crossing will produce a closed curve $\xi$ which is either contractible closed curve or which encloses $\star$. 
In the former case, we set $f(\xi) = -2$ and in the latter we set $f(\xi) = \lambda_p$ where $p$ is the order of $\star$. 
In particular, in the case of $P_n^\star$, performing this process for two distinct pending arcs, which will have two points of intersection, yields 
\begin{equation}
  f(v_i,v_i) f(v_j,v_j) = f(v_i,v_j)^2 + \lambda_p f(v_i,v_j)f(v_j,v_i) + f(v_j,v_i)^2 . 
\label{eq:GenMutat}
\end{equation}

\begin{definition}
Given an integral domain $R$, a \emph{frieze} on a surface or orbifold $(S,M)$  is a function $f: \{\text{arcs on }(S,M)\} \to R$ such that 
\begin{itemize}
    \item $f(\gamma) = 1$ if $\gamma$ is isotopic to a boundary segment and
    \item $f$ respects skein relations.
\end{itemize}
If all values of $f$ lie in $\mathbb{Z}_{>0}$, we say $f$ is a \emph{positive integral frieze}.
\end{definition}

An important question in the study of friezes is whether every frieze over a fixed surface or orbifold is \emph{unitary}. 
A frieze on $(S,M)$ is unitary if there exists a triangulation $T$ such that $f(\tau) = 1$ for all $\tau \in T$. Since this uniquely determines the frieze, we denote this $f_T$. Fontaine and Plamondon show that friezes on $P_n^\bullet$ are not always unitary \cite{FontainePlamondon}. 
Felikson and Tumarkin recently showed that every frieze on an unpunctured surface is unitary and classified friezes for general surfaces \cite{felikson2024friezes}. 

If $\star$ is order 2, then standard arcs on $P_n^\star$ correspond to centrally symmetric pairs of arcs on $P_{2n}$ while pending arcs correspond to diameters $(v_i,v_{n+i})$. Centrally symmetric triangulations of $P_{2n}$ provide a geometric realization for cluster algebras of type $C_{n-1}$. Using the technique of folding Dynkin diagrams, Fontaine and Plamondon showed the following. 

\begin{proposition}[\cite{FontainePlamondon}]\label{prop:OrbifoldFriezesUnitaryOrder2}
Given a positive integral orbifold frieze $f$ on $P_n^\star$ where $\star$ is an orbifold point of order 2, there exists a unique triangulation $T$ of $P_n^\star$ such that for all $\tau \in T$, $f(\tau) = 1$. 
\end{proposition}

\begin{proof}
A frieze on $P_n^\star$ when $\star$ has order 2 is equivalent to a frieze of type $C_{n-1}$. Fontaine and Plamondon show that such friezes are in 1-1 correspondence with friezes on $P_{2n}$ associated to centrally symmetric triangulations \cite{FontainePlamondon}. Equivalently, each frieze on $P_n^\star$ is of the form $f_T$ for $T$ a triangulation of $P_n^\star$.
\end{proof}

The case of order 2 is special because $\lambda_2 = 0$, so one term in \Cref{eq:GenMutat} disappears. 
The natural next candidate to consider is order 3. Our proof in this case is inspired by the proof for friezes on $P_n$, which is given in \cite[Question 25]{conway1973triangulated,conway1973answers}.

\begin{proposition}
\label{prop:OrbifoldFriezesUnitaryOrder3}
Given a positive integral orbifold frieze $f$ on $P_n^\star$ where $\star$ is an orbifold point of order 3, there exists a unique triangulation $T$ of $P_n^\star$ such that for all $\tau \in T$, $f(\tau) = 1$. 
\end{proposition}

\begin{proof} 
We will induct on $n$. 
The case of $P_1^\star$ is trivial because there are no non-boundary arcs on $P_1^\star$. We also directly check $P_2^\star$. 
Here, there are two non-boundary arcs, $(v_0,v_0)$ and $(v_1,v_1)$ whose values under any frieze $f$ satisfy $f(v_0,v_0) f(v_1,v_1) = f(v_0,v_1)^2 + f(v_0,v_1)f(v_1,v_0) + f(v_1,v_0)^2 = 3$. 
We see that we must have either $f(v_0,v_0) = 1$ and $f(v_1,v_1) = 3$ or the same with the roles of $v_0$ and $v_1$ swapped. 
In each case, the single arc with value 1 under the frieze forms a triangulation of $P_2^\star$. 

Now, we assume we have shown our claim for $P_{n-1}^\star$ and we consider a frieze on $P^\star_n$. 
Assume for sake of contradiction that for all $0 \leq i \leq n-1$, $f(v_{i-1},v_{i+1}) > 1$. For all $2 \leq i \leq n-1$, we have $f(v_0,v_i)f(v_{i-1},v_{i+1}) = f(v_0,v_{i-1}) + f(v_0,v_{i+1})$. 
In other words, 
\[
f(v_0,v_{i+1}) = f(v_0,v_i) f(v_{i-1},v_{i+1}) - f(v_0,v_{i-1}) \geq 2 f(v_0,v_i) - f(v_0,v_{i-1}).
\]

Similarly, the skein relation $f(v_0,v_{n-1}) f(v_{n-2},v_0) = f(v_0,v_0) + f(v_0,v_{n-2})$ yields 
\[
f(v_0,v_0) = f(v_0,v_{n-1}) f(v_{n-2},v_0) - f(v_0,v_{n-2}) \geq 2 f(v_0,v_{n-1}) - f(v_0,v_{n-2}).
\]
Therefore, we have a chain of inequalities, 
\[
f(v_0,v_0) - f(v_0,v_{n-1}) \geq f(v_0,v_{n-1}) - f(v_0,v_{n-2}) \geq \cdots \geq f(v_0,v_{2}) - f(v_0,v_1) \geq 2-1 > 0
\]
implying 
\[
f(v_0,v_0) > f(v_0,v_{n-1}) > \cdots > f(v_0,v_2) > 1.
\]

We can replace $v_0$ with any other vertex and get a similar chain of inequalities.
Now, suppose we have indexed the vertices so that $f(v_0,v_0) = \min\{f(v_j,v_j) : 0 \leq j \leq n-1\}$. Apply the skein relation from \Cref{eq:GenMutat} to the intersections between $(v_0,v_0)$ and $(v_{n-1},v_{n-1})$, 
\begin{align*}
    f(v_0,v_0)f(v_{n-1},v_{n-1}) 
    & = f(v_0,v_{n-1})^2 + f(v_0,v_{n-1})f(v_{n-1},v_0) + f(v_{n-1},v_0)^2 \\
    & = f(v_0,v_{n-1})^2 + f(v_0,v_{n-1}) + 1 
\end{align*}
By our designation of $v_0$, we have 
\begin{equation}
    f(v_0,v_{n-1})^2 + f(v_0,v_{n-1}) + 1 \geq  f(v_0,v_0)^2 .
\label{eq:EndOfUnitaryProof}
\end{equation}
Recall we have shown $f(v_0,v_0) > f(v_0,v_{n-1})$. 
Since $n \geq 3$, we have that $f(v_0,v_0) \geq 3$ and $f(v_0,v_{n-1}) \geq 2$, and we can conclude that \Cref{eq:EndOfUnitaryProof} is impossible. 

From this contradiction, we see there must exist a value $0 \leq i \leq n-1$ such that $f(v_{i-1},v_{i+1})$ is 1. 
Restricting to values $f(v_j,v_k)$ for $0 \leq j,k \leq n-1, j\neq i \neq k$ gives a frieze on $P_n^\star$ where the edge $(v_{i-1},v_{i+1})$ behaves as a boundary arc. 
By the inductive hypothesis, there exists a triangulation $T'$ of this copy of $P_{n-1}^\star$ such that for all $\tau \in T'$, $f(\tau) = 1$. 
Therefore, $T:= T' \cup \{(v_{i-1},v_{i+1})\}$ is also such a triangulation $f$, showing $f$ is unitary. 
\end{proof}

\begin{example}

Let $T$ be the triangulation of $P_3^\star$ given below.

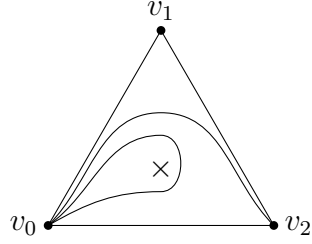
\begin{center}
\begin{tikzpicture}[scale=1.5]
\draw(0,0) -- (2,0) -- (1,1.73) -- (0,0);
\node[] at (1,0.5){$\times$};
\draw(0,0) to [out = 45, in = 180] (1,1);
\draw(1,1) to [out = 0, in = 135]  (2,0);
\draw(0,0) to [out = 35, in = 180] (1,0.8);
\draw(1,0.8) to [out = 0, looseness = 1.2,in = 0] (1,0.3);
\draw(1,0.3) to [out = 180, in = 30] (0,0);
\node[circle, fill = black, scale = 0.3] at (0,0)    {};
\node[circle, fill = black, scale = 0.3] at (2,0)    {};
\node[circle, fill = black, scale = 0.3] at (1,1.73) {};
\node[left]  at (0,0)    {$v_0$};
\node[above] at (1,1.73) {$v_1$};
\node[right] at (2,0)    {$v_2$};
\end{tikzpicture}
\end{center}

We display the entries of the frieze associated to this triangulation in the table below for $\star$ both order 2 and order 3, where the entry in row $v_i$ and column $v_j$ is $f(v_i,v_j)$.

\vspace{0.4cm}
\begin{center}
\begin{tabular}{|l|l|l|l|}
\hline
\rowcolor{light-blue} 
\multicolumn{4}{|c|}{$\star$ has order $2$}  \\
\hline 
\rowcolor{grayish} 
\footnotesize $v_i \backslash v_j$ & $v_0$ & $v_1$ & $v_2$ \\
\hline
\cellcolor{grayish} $v_0$ & 1 & 1 & 1 \\
\cellcolor{grayish} $v_1$ & 2 & 5 & 1 \\
\cellcolor{grayish} $v_2$ & 1 & 3 & 2 \\
\hline
\end{tabular}
\hspace{0.8cm}
\begin{tabular}{|l|l|l|l|}
\hline
\rowcolor{light-blue} 
\multicolumn{4}{|c|}{$\star$ has order $3$}  \\
\hline 
\rowcolor{grayish} 
\footnotesize   $v_i \backslash v_j$ & $v_0$ & $v_1$ & $v_2$ \\
\hline
\cellcolor{grayish} $v_0$ & 1 & 1 & 1 \\
\cellcolor{grayish} $v_1$ & 2 & 7 & 1 \\
\cellcolor{grayish} $v_2$ & 1 & 4 & 3 \\
\hline
\end{tabular}
\end{center}

 One can check that \Cref{eq:GenMutat} holds for these values as well as other skein relations. For instance, when $\star$ is order 3, we see \[
 f(v_1,v_1) f(v_2,v_2) = 7 \times 3 = 4^2 + 1 \times 4 \times 1 + 1^2 = f(v_2,v_1)^2 + \lambda_3 f(v_2,v_1)f(v_1,v_2) + f(v_1,v_2)^2.\]
\end{example}

Knowing that friezes on $P_n^\star$ are unitary when $\star$ is order 2 or 3 allows us to identify each with a finite positive integral frieze pattern. 

\begin{corollary}~

\begin{enumerate}
    \item There is a 1-1 correspondence between positive integral frieze patterns with width $2n$ and principal growth coefficient 0 and friezes on $P_n^\star$ where $\star$ has order 2.
    \item There is a 1-1 correspondence between positive integral frieze patterns with width $3n$ and principal growth coefficient 1 and friezes on $P_n^\star$ where $\star$ has order 3.
\end{enumerate}
\end{corollary}

\begin{proof}
By \Cref{prop:SortFriezeByGrCo}, finite frieze patterns with principal growth coefficient 0 are in 1-1 correspondence with centrally symmetric triangulations, and these are in correspondence with triangulations of $P_n^\star$ where $\star$ is order 2. By \Cref{prop:OrbifoldFriezesUnitaryOrder2}, these triangulations are in 1-1 correspondence with friezes on $P_n^\star$. This shows statement (1), and statement (2) is analogous, where we instead use \Cref{prop:OrbifoldFriezesUnitaryOrder3}.
\end{proof}

\section{Frieze Patterns and Dyck Paths}
\label{sec: frieze patterns and dyck paths}

We have so far explored correspondences between dissections of polygons and frieze patterns with a particular eye towards the compatibility of the symmetries of each object. In this section, we will export this narrative  to the setting of Dyck paths. In particular, we will highlight the action on $m$-Dyck paths induced by rotation of $(m+2)$-angulations of polygons. 


A Dyck path is a certain type of lattice path. These are a well-known family of Catalan objects and hence are equinumerous  with triangulations of polygons and finite, positive integral frieze patterns. In order to discuss  $(m+2)$-angulations for $m > 1$, we will recall the more general notion of $m$-Dyck paths.

\begin{definition}
\label{def: m-dyck paths}
An \emph{$m$-Dyck path of order $\ell$} is a sequence of steps in direction $(0,1)$ (denoted $U$) or $(1,0)$ (denoted $R$) from $(0,0)$ to $(m\ell,\ell)$ which stays above the $y = \frac{1}{m} x$ diagonal line. 
\end{definition}

See the left-hand side of \Cref{fig:prop-2:etherington-to-dyck-path} for an example of a 2-Dyck path. We remark that  it is natural to conflate the lattice paths of an $m$-Dyck path with the $m$-Dyck path itself, and we will do this at times. 

When $m = 1$, \Cref{def: m-dyck paths} recovers that of standard Dyck paths. In general, the number of $m$-Dyck paths of order $\ell$ is equal to the Fuss-Catalan number $c_\ell^{(m)}$. 

Let $D = w_1,\ldots,w_{d}$ be an $m$-Dyck path; i.e., each $w_i \in \{U,R\}$. It is useful to note that the property of staying above the line $y = \frac{1}{m}x$ is equivalent to the property that the number of $R$'s in any initial subsequence $w_1,\ldots,w_k$ is no greater than $m$ times the number of $U$'s. Such a sequence can be referred to as \emph{$m$-ballot} or \emph{$m$-Yamaouchi}.

Our first goal is to recall a bijection between $m$-Dyck paths of order $\ell$ and $(m+2)$-angulations of $P_{m\ell+2}$, described in \cite{etherington1940some}; see also the description in \cite{stanley1996polygon}.

Recall the vertices of $P_n$ are labeled $v_0,\ldots,v_{n-1}$ in clockwise order. 
Let $\mathcal{T}^{(m)}_{m\ell + 2}$ denote the set of $(m+2)$-angulations of $P_{m\ell + 2}$ and let $\mathcal{D}^{(m)}_{\ell}$ denote the set of $m$-Dyck paths of order $\ell$.

\begin{definition}
\label{def: brow-bijection}
We define the \emph{browse} map $\brow: \mathcal{T}_{m\ell + 2}^{(m)} \to \mathcal{D}^{(m)}_{\ell}$ by associating to each $T \in \mathcal{T}_{m\ell + 2}^{(m)}$ the word $\brow(T)$ determined by performing the following steps in the prescribed order.
Walk clockwise around $P_{m\ell+2}$, beginning at $v_0$ and visiting each vertex. During the \emph{visit} of $v_i$, survey the incident subgons, sweeping in counterclockwise order. 
There are three possible cases when we encounter a subgon $Q$, and we outline the effect to the word in \Cref{fig:Etherington-seq-cases}.
\end{definition}

\begin{figure}[ht]
    \centering
\scalebox{0.95}{
    \begin{tabular}{|l|l| } 
    \hline
    \rowcolor{light-blue}
    \footnotesize Case for an $(m+2)$-gon $Q$ incident to a vertex of $P_n$ 
    & \footnotesize  Modification to the word \\
    \hline
    \footnotesize \textbf{Case 1:} This is the first time we see $Q.$       
    & \footnotesize Append a $U$. \\ 
    \footnotesize \textbf{Case 2:} This neither the first time nor the $(m+2)$-th time we see $Q.$     
    & \footnotesize Append an $R$. \\ 
    \footnotesize \textbf{Case 3:} This is the $(m+2)$-th time we see $Q.$     
    & \footnotesize  Append nothing. \\
    \hline
    \end{tabular} }
    \caption{An outline of the choices for building an $m$-Dyck path from an $(m+2)$-angulation of a polygon. }
    \label{fig:Etherington-seq-cases}
\end{figure}

Using the algorithm in \Cref{def: brow-bijection} yields the $2$-Dyck path on the right of \Cref{fig:prop-2:etherington-to-dyck-path} from the 4-angulation on the left. For instance, $v_0$ is incident to two subgons and naturally this is the first time we seen each subgon, hence the Dyck path begins with two up steps. Then, $v_1$ is incident to one of these subgons which was already viewed by $v_0$ and a new subgon, so the path continues with a right step and another up step.

The fact that the word $\brow(T)$ indeed determines an $m$-Dyck path was proven in \cite{etherington1940some}.
We will next define a map which will be shown to be the inverse to the browse map. First, we introduce some terminology and conventions for $m$-Dyck paths.
Call a lattice point on a Dyck path a \emph{corner} if it is incident to both an up step and a right step and let all other lattice points be \emph{non-corners}. 

\begin{definition}
\label{def: balance lines}
    Given an $m$-Dyck path of order $\ell$, create a multiset of size $\ell-1$ of lines of slope $\frac{1}{m}$, with one passing through the lowest point of each up step, excluding the first up step. The lines in the multiset are called \emph{balance lines}.\footnote{These are quite similar to paths considered in \cite{bergeron2011higher}.} 
\end{definition}

There are four balance lines drawn with dotted lines on the left in \Cref{fig:prop-2:etherington-to-dyck-path}, where the middle one comes with multiplicity 2 as it goes through two corners. The teal balance line and its first intersection with a non-corner will be of extra importance in \Cref{thm: rotation of n-gon and shifted Dyck path}, hence the star at its end.


Let $l$ be a balance line of an $m$-Dyck path. If $l$ begins at point $(i,i')$ and first intersects a non-corner of $D$ at $(j,j')$ then we \emph{label} $l$ with the pair $(i,j+1)$. For instance, the teal balance line in \Cref{fig:prop-2:etherington-to-dyck-path} is labeled $(0,9)$, the two red balance lines are labeled $(1,8)$ and $(3,8)$, and the orange balance line is labeled $(3,6)$.  One subtle point worth emphasizing is that each up step which is incident to the diagonal $y = \frac{1}{m}x$ has its first intersection with a non-corner at the terminal point $(m\ell,\ell)$. Therefore, each such balance line is labeled $(i,m\ell + 1)$ where $i$ is necessarily positive by our convention in \Cref{def: balance lines}.

\begin{definition}\label{def:rtn2}
We define the \emph{return} map $\rtn: \mathcal{D}^{(m)}_{\ell} \to  \mathcal{T}_{m\ell + 2}^{(m)}$ by associating to each $m$-Dyck path $D$ the $(m+2)$-angulation $\rtn(w)$ given by diagonals $\{(v_i,v_j)\}$ where $\{(i,j)\}$ is the set of labels of balance lines of $D$. 
\end{definition}

The following also implies that if $D$ is an $m$-Dyck path of order $\ell$, then $\rtn(D) \in \mathcal{T}^{(m)}_{m\ell+2}$.

\begin{proposition}
\label{prop:EthSequenceTom-Angulation}
The maps $\brow$ and $\rtn$ are inverse bijections.
\end{proposition}

\begin{proof}

The fact that $\brow$ is bijective follows from Etherington \cite{etherington1940some}. Note that one can translate between the construction therein and \Cref{def: brow-bijection} by using an $m$-generalization of the usual bijection between parenthesizations and 
Dyck paths.

Showing $\brow$ and $\rtn$ are inverse maps is routine. The key step is the following. Given $T \in \mathcal{T}^{(m)}_n$, suppose $(i,i')$ is the base of an up step in $D = \brow(T)$ stemming from seeing subgon $Q$. The path $D$ will only pass back under the balance line which begins at $(i,i')$ in a step that corresponds to a visit occurring after we have seen $Q$ all $m+2$ times. 
 
\end{proof}

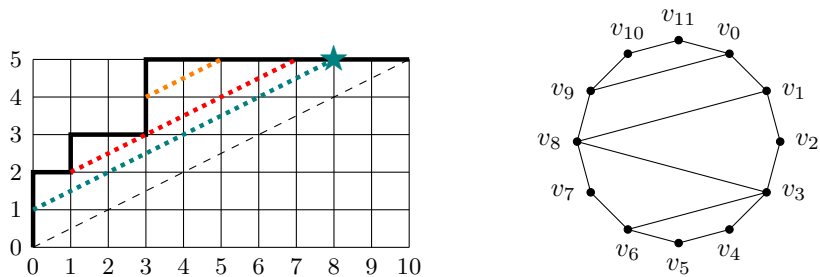
\begin{figure}[ht]
    \centering
\begin{tabular}{cc}
\begin{tikzpicture}[scale=0.5]
    \foreach \i in {0,...,10} {
        \draw [thin,black] (\i,0) -- (\i,5)  node [below] at (\i,0) {\scriptsize $\i$};
    }
    \foreach \i in {0,...,5} {
        \draw [thin,black] (0,\i) -- (10,\i) node [left] at (0,\i) {\scriptsize $\i$};
    }
    
    \draw[dashed] (0,0) -- (10,5);
    \draw[line width=0.6mm] (0,0) -- (0,2) 
                                  -- (1,2) 
                                  -- (1,3) 
                                  -- (3,3) 
                                  -- (3,5) 
                                  -- (10,5);
    \draw[line width=0.6mm, dotted, teal]   (0,1) -- (8,5);
    \draw[line width=0.6mm, dotted, red]    (1,2) -- (7,5);
    \draw[line width=0.6mm, dotted, orange] (3,4) -- (5,5);
    \node at (8,5) {$\textcolor{teal}{\bigstar}$};
\end{tikzpicture}
&
\begin{tikzpicture}[scale = 0.5]
 \newdimen\R
    \R=2.7cm
    
 \draw (30:\R) \foreach \x in {60,90,...,390} {  -- (\x:\R) };  
    \foreach \x/\l/\p in
     {
      60/{\footnotesize  $v_0$}/above,
      90/{\footnotesize  $v_{11}$}/above,
      120/{\footnotesize  $v_{10}$}/above,
      150/{\footnotesize  $v_9$}/left,
      180/{\footnotesize  $v_8$}/left,
      210/{\footnotesize  $v_7$}/left,
      240/{\footnotesize  $v_6$}/below,
      270/{\footnotesize  $v_5$}/below,
      300/{\footnotesize  $v_4$}/below,
      330/{\footnotesize  $v_3$}/right,
      360/{\footnotesize  $v_2$}/right,
      390/{\footnotesize $v_1$}/right
     }
   \node[inner sep=1pt,circle,draw,fill,label={\p:\l}] at (\x:\R) { };

   \draw(60:\R) -- (150:\R);
   \draw(30:\R) -- (180:\R) -- (-30:\R)--(240:\R);
\end{tikzpicture}\\
\end{tabular}
    \caption{On the left, we have a 2-Dyck path of order 5 and on the right, we have the corresponding 4-angulation of a 12-gon. These are related by the maps $\brow$ and $\rtn$. }
    \label{fig:prop-2:etherington-to-dyck-path}
\end{figure}

The definition of the return map suggests how one could directly construct a frieze pattern from a Dyck path. In fact, this was recently done in the ordinary case (i.e. $m = 1$) by Ca{\~n}adas,  Espinosa, Gaviria, and Rios \cite{canadas2023coxeter}. Our result, for general $m$, recovers theirs when $m = 1$. First, we give two statistics on $m$-Dyck paths. 

\begin{definition}
\label{def: num of balance-lines and num of up-steps}
    Let $D$ be an $m$-Dyck path. 
    For each $0 \leq i \leq m\ell + 1$, let $\text{up}_D(i)$ denote the number of up steps of the form $(i,j) - (i,j+1)$, $j > 0$, in $D$, and let $\text{bal}_D(i)$ (short for balance) denote the number of times a balance line has its first intersection with a non-corner point at a point $(i-1,j)$.
\end{definition}

 For example, in the Dyck path in \Cref{fig:prop-2:etherington-to-dyck-path}, $\text{up}_D(0) = \text{up}_D(1) = 1$,  $\text{up}_D(3) = 2$, $\text{bal}_D(6) = \text{bal}_D(9) = 1$, and $\text{bal}_D(8) = 2$. Since our Dyck path is defined in the $m\ell \times \ell$ rectangle, $\text{up}_D(m\ell + 1)$ and $\text{bal}_D(0)$ are always  $0$. Our convention on balance lines also implies that $\text{bal}_D(1)$ is always  $0$.

\begin{proposition}
\label{prop:FrizeFromDyck}
Let $D$ be an $m$-Dyck path of order $\ell$. The frieze pattern $F_{\rtn(D)}$ satisfies 
\[
F_{\rtn(D)}(i-1,i+1) = (\text{up}_D(i) + \text{bal}_D(i)+1) \lambda_{m+2}
\]
for all $0 \leq i \leq  m\ell + 1$.

 \end{proposition}

 \begin{proof}
Set $T:= \rtn(D)$. Since $T$ is an $(m+2)$-angulation, following \Cref{def:FriezePatternsFromDissections}, the value $F_T(i-1,i+1)$ is $(\deg_T(v_i)+1)\lambda_{m+2}$ where $\deg_T(v_i)$ is the number of diagonals of $T$ incident to $v_i$. There is one diagonal based at $v_i$ for each balance line which is labeled $(i,j)$ or $(j,i)$. The statistic $\text{up}_D(i)$ counts the number of the first type of balance line and the statistic $\text{bal}_D(i)$ counts the number of the scond type of balance line. 
\end{proof}

In particular, combining \Cref{thm:HJ} and \Cref{prop:EthSequenceTom-Angulation}, we have a bijection between $m$-Dyck paths of order $\ell$ and frieze patterns of type $\Lambda_{m+2}$ with width $m\ell + 2$. For example, the 2-Dyck path on the left in \Cref{fig:prop-2:etherington-to-dyck-path} is in correspondence with the frieze pattern in \Cref{fig:FriezePatternFromDyckPath}. It would be interesting to find deeper connections between these combinatorial objects.

\begin{quest}
Given an $m$-Dyck path $D$, how can one interpret $F_{\rtn(D)}(i,j)$ for general $i < j$ in terms of $D$?
\end{quest}

\begin{figure}
\scalebox{0.75}{
\centering
$\begin{array}{ccccccccccccccccccccccccc}
   &0&&  0&&  0&&  0&&  0&&  0&&  0&&  0 &&0&&  0&&  0&&  0&\\
  &&  1&&  1&&  1&&  1&&  1&&  1&&  1&&1&&  1&&  1&&  1&&  1\\
 &2\sqrt{2}&&2\sqrt{2}&&\sqrt{2}&&  3\sqrt{2}&&  \sqrt{2}&&  \sqrt{2}&&  2\sqrt{2}&&\sqrt{2}&& 3\sqrt{2} && 2\sqrt{2} && \sqrt{2} && \sqrt{2} &\\
  &&7&&  3&&  5&& 5&&  1&&  3&&3&&5&& 11&&3&&1&&3\\ &5\sqrt{2}&&5\sqrt{2}&&7\sqrt{2}&&  4\sqrt{2}&&  2\sqrt{2}&&  \sqrt{2}&&  2\sqrt{2}&&7\sqrt{2}&& 9\sqrt{2} && 8\sqrt{2} && \sqrt{2} && \sqrt{2} &\\
  &&7&&  23&&  11&& 3&&  3&&  1&&9&&25&& 13&&5&&1&&3\\
  &2\sqrt{2}&&16\sqrt{2}&&18\sqrt{2}&&  4\sqrt{2}&&  2\sqrt{2}&&  \sqrt{2}&&  2\sqrt{2}&&16\sqrt{2}&& 18\sqrt{2} && 4\sqrt{2} && 2\sqrt{2} && \sqrt{2} &\\
  &&9&&  25&&  13&& 5&&  1&&  3&&7&&23&& 11&&3&&3&&1\\
& 2\sqrt{2}&&7\sqrt{2}&& 9\sqrt{2} && 8\sqrt{2} && \sqrt{2} && \sqrt{2} &&5\sqrt{2}&&5\sqrt{2}&&7\sqrt{2}&&  4\sqrt{2}&&  2\sqrt{2}&&  \sqrt{2}&\\
&&3&&5&& 11&&3&&1&&3  &&7&&  3&&  5&& 5&&  1&&  3\\ 
&\sqrt{2}&& 3\sqrt{2} && 2\sqrt{2} && \sqrt{2} && \sqrt{2} &&2\sqrt{2}&&2\sqrt{2}&&\sqrt{2}&&  3\sqrt{2}&&  \sqrt{2}&&  \sqrt{2}&&  2\sqrt{2}&\\
  &&  1&&  1&&  1&&  1&&  1&&  1&&  1&&1&&  1&&  1&&  1&&  1\\
   &0&&  0&&  0&&  0&&  0&&  0&&  0&&  0 &&0&&  0&&  0&&  0&\\
 \end{array}$
}
\caption{The frieze pattern of type $\Lambda_4$ associated to the 4-angulation in \Cref{fig:prop-2:etherington-to-dyck-path}. One could also regard this as a frieze pattern associated to the 2-Dyck path in the same figure.}\label{fig:FriezePatternFromDyckPath}
\end{figure}

We now have three families of objects in bijection: $(m+2)$-angulations, finite frieze patterns of type $\Lambda_{m+2}$, and $m$-Dyck paths. In \Cref{lem:ShiftAndRotate}, we saw the effect that cyclic rotation of an $(m+2)$-angulation induced on a frieze pattern. In \Cref{thm: rotation of n-gon and shifted Dyck path}, we will do the same for $m$-Dyck paths. We first define the proposed action. 

\begin{definition}
\label{def: rotation of an n-gon and shifted Dyck path}
Define the map
\begin{align*}
    \tilderot:\calD^{(m)}_\ell \to \calD^{(m)}_\ell
\end{align*}
where for $D\in\calD^{(m)}_\ell$ the $m$-Dyck path, $\tilderot(D),$ is obtained by the following process:
    \begin{enumerate}
        \item Let $k+1$ be the position of the first $R$ in $D$.
        \item Let $p_1,\ldots,p_{k-1}$ be the first non-corner intersection points of the balance lines with labels $(0,j)$. Append a $U$ after each $R$ step which begins at a point $p_i$. 
        \item Delete the first $k+1$ letters in $D$. This sequence will be $U^{k}R.$
        \item Prepend a $U$ and append an $R.$
    \end{enumerate}
\end{definition}

\begin{remark}\label{rmk: rotation of an n-gon and shifted Dyck path rounds2}
It can be convenient to rephrase part (2) of \Cref{def: rotation of an n-gon and shifted Dyck path} without appealing to the lattice path representing $D$. 
Let $D = w_1,\ldots,w_d$ be an $m$-Dyck path and let $k+1$ be the smallest number such that $w_{k+1} = R$. Define $h: \{U,R\} \to \mathbb{Z}$ by $h(U) = m$ and $h(R) = -1$, and define the \emph{height sequence} $H_D: [d] \to \mathbb{Z}$ by $H_D(j) = \sum_{i=1}^j h(w_i)$. The $m$-ballot property guarantees that $H_D(j) \geq 0$ for all $j \in [d]$. 

Then, the definition of $\tilderot$ can be rephrased by replacing (2) with (2') below:
    \begin{enumerate}
        \item[(2')] For each $1 \leq i \leq k-1$, let $p_i'>k$ be the smallest integer such that $H_D(p_i') < mi$. Form a length $d+(k-1)$ binary word by placing a $U$ at all positions $p_i' + k-i$ and filling out the remaining positions with the given Dyck word, in the same order.
    \end{enumerate}

\end{remark}

\begin{remark}
    One might expect that, when $m = 1$, $\rottilde$ coincides with \emph{promotion} on Dyck paths. Such an action can be defined by the induced action of promotion on $2 \times n$ Standard Young Tableaux (SYT), and therefore this is another cyclic action on Dyck paths  \cite{Haiman}. One can quickly check that  these actions do not agree in general though. For instance, applying promotion in this sense to the 1-Dyck path $UUURRUR$ yields $UURRUURR$ whereas applying $\rottilde$ yields $URURUURR$ and applying $\rottilde^{-1}$ yields $UURUURRR$.
\end{remark}

Given $D \in \mathcal{D}_\ell^{(m)}$, consider the following operation.

\begin{enumerate}
    \item Let $z$ be the number of times $D$ intersects the line $y = \frac1m x$, including its start point but not its end point. 
    \item Delete each of the $z$ $U$'s whose starting point is on the line $y = \frac1m x$. 
    \item Delete the final $R$.
    \item Prepend $U^zR$.
\end{enumerate}

One can see that the above procedure is the inverse operation to $\rottilde$. In particular, $\rottilde$ is a bijection. One can rephrase this procedure, i.e. $\rottilde^{-1}$, in terms of words by noting that intersections of $D$ with the line $y = \frac1m x$ coincide with indices where the height sequence $H_D$ is 0.

Our final goal is to show that $\rottilde$ is the map on Dyck paths induced by rotation of $m$-angulations. Recall $\sigma$ denotes counterclockwise rotation by $2\pi/n$ for an $n$-gon.

\begin{theorem}
\label{thm: rotation of n-gon and shifted Dyck path}
    The following diagram commutes for all $\ell > 0$:
    \begin{center}
    \begin{tikzpicture}[>=Stealth, auto]
    \node (B) at (-2.25, 2)  {$\calT^{(m)}_{m\ell+2}$};
    \node (C) at (2.25, 2)  {$\calD^{(m)}_\ell$};


    \node (Y) at (-2.25, 0)  {$\calT^{(m)}_{m\ell+2}$};
    \node (Z) at (2.25, 0)  {$\calD^{(m)}_\ell$};

    \draw[->] (B) -- (C) node[midway, above] {$\brow$};


    \draw[->] (Y) -- (Z) node[midway, below] {$\brow$};
    
    \draw[->] (B) -- (Y) node[midway, left] {$\rott$};
    \draw[->] (C) -- (Z) node[midway, right] {$\widetilde{\mathsf{rot}}$};
    \end{tikzpicture}
    \end{center}
\end{theorem}

\begin{proof} 
Let $T \in \mathcal{T}^{(m)}_{m\ell + 2}$. For each subgon $Q$ cut out by $T$, let $v(Q)$ be the smallest indexed vertex incident to $Q$. Label the subgons incident to $v_0$ by $Q_k,\ldots,Q_1$ in clockwise order, i.e., so that this indexing is opposite from the order of the visits of the browse map. In particular, $v_{n-1}$ is a vertex on $Q_k$ and $v_1$ is a vertex on $Q_1$. 
Let $r_i$ be such that $v_{r_i+1}$ is the vertex with second smallest index incident to $Q_i$. In particular, $r_1 = 0$, and for $i > 1$, $(v_0,v_{r_i+1}) \in T$.

An $m$-Dyck path $D \in \mathcal{D}^{(m)}_\ell$ is determined by the number of up steps on each line $x = i$, i.e, the values $\text{up}_D(i)$ for all $0 \leq i \leq m\ell + 1$. From the definition of the browse map, we see that $\text{up}_{\brow(T)}(i)$ is  equal to the number of subgons $Q$ cut out by $T$ satisfying $v(Q) = v_i$. If we instead apply $\brow$ to the counterclockwise rotation of $T$, we have \[
\text{up}_{\brow(\rott(T))}(i) = \begin{cases}  \text{up}_{\brow(T)}(i+1) & i \neq r_j \text{ for all }j\\
\text{up}_{\brow(T)}(i+1) + 1 & i = r_j \text{ for some value } j.
\end{cases}
\]

Now, we want to show that $\text{up}_{\rottilde(\brow(T))}$ has the same relationship with $\text{up}_{\brow(T)}$. Since the $\rottilde$ map deletes one initial $R$, we see that $\text{up}_{\rottilde(\brow(T))}(i) = \text{up}_{\brow(T)}(i+1)$ if there are no points $p_j$ on the line $x = i$ in $\brow(T)$ and otherwise $\text{up}_{\rottilde(\brow(T))}(i) = \text{up}_{\brow(T)}(i+1) + 1$. 

In $\brow(T)$, the intersection points $p_j$ occur at the beginning of $R$ steps associated to visiting the subgon $Q_{j+1}$ at vertex $v_{r_j+1}$, i.e.,  visiting this subgon for the second time. This is true because the height sequence will only reduce by more than $m$ after our first encounter with subgon $Q_j$ once we have  an intermediate visit at $Q_{j+1}$. Therefore, the condition ``there are no points $p_j$ on the line $x = i$'' is equivalent to ``$i \neq r_j$ for all $j$'', and similarly for the negation. We conclude $\rottilde(\brow(T)) = \brow(\rott(T))$ since these $m$-Dyck paths have the same number of up steps along each vertical line. 
\end{proof}

\begin{example}
    Let $T$ be the 4-angulation displayed on the right in \Cref{fig:prop-2:etherington-to-dyck-path} and let $D$ be the 2-Dyck path displayed in the same figure on the left. 
   Below, we reproduce both the Dyck path (written as a word in $U$'s and $R$'s) and the values of the height function $H_D:[d]\to\mathbb{Z}$ as defined in \Cref{rmk: rotation of an n-gon and shifted Dyck path rounds2}.
    \[
    \begin{array}{cccccccccccccc}
        & \text{Dyck path:}
        & \overbrace{U~U}^{v_0}~
        & \overbrace{R~U}^{v_1}~
        & \overbrace{R}^{v_2}~
        & \overbrace{R~U~U}^{v_3}~
        & \overbrace{R}^{v_4}~
        & \overbrace{R}^{v_5}~
        & \overbrace{R}^{v_6}~
        & \overbrace{R}^{v_7}~
        & \overbrace{R}^{v_8}~
        & \overbrace{R}^{v_9}~
        & \overbrace{R}^{v_{10}} \\
        & H_D(j):
        & 2\hspace{0.25cm} 4
        & 3\hspace{0.25cm} 5
        & 4
        & 3\hspace{0.25cm} 5\hspace{0.25cm} 7
        & 6
        & 5
        & 4
        & 3
        & 2
        & 1
        & 0
    \end{array}
    \]

    In  \Cref{fig:prop-2:etherington-to-dyck-path-rotated}, $\rott(T)$ is drawn on the right and $\brow(\rott(T))$ is drawn on the left. We will demonstrate that $\brow(\rott(T)) = \rottilde(D)$. The value $k$ is 2 since the third step of $D$ is the first right step. The intersection point $p_1$ is drawn as a star in \Cref{fig:prop-2:etherington-to-dyck-path}. In the language of \Cref{rmk: rotation of an n-gon and shifted Dyck path rounds2}, $p_i'>2$ is the smallest index such that $H(p_i') < 2$. This occurs at the second-to-last position, i.e., $p_i' = 14$. Therefore, the 2-Dyck path $\rottilde(D)$ is the result of adding a $U$ after the 14th step, deleting the initial subword $U^2~R$ but then reinserting an initial $U$, and appending a final $R$. The additions are both underlined and colored in orange below.
    \[
    \begin{array}{ccccccccccccccccccccc}
         & ~ & U & U & R & ~ & U & R & R & U & U & R & R & R & R & R & \textcolor{teal}{\underline{R}} & ~ & R & ~ \\
         & \mapsto & \cancel{U} & \cancel{U} & \cancel{R} & \textcolor{orange}{\underline{U}} 
         & U & R & R & U & U & R & R & R & R & R & \textcolor{teal}{\underline{R}} & \textcolor{orange}{\underline{U}} & R & \textcolor{orange}{\underline{R}}
    \end{array}
    \] 
\end{example}

\begin{figure}[ht]
    \centering
\begin{tabular}{cc}
\begin{tikzpicture}[scale=0.5]
    \foreach \i in {0,...,10} {
        \draw [thin,black] (\i,0) -- (\i,5)  node [below] at (\i,0) {\scriptsize $\i$};
    }
    \foreach \i in {0,...,5} {
        \draw [thin,black] (0,\i) -- (10,\i) node [left] at (0,\i) {\scriptsize $\i$};
    }
    
    \draw[line width=0.6mm, color=black] (0,1) -- (0,2) 
                                               -- (2,2) 
                                               -- (2,4)
                                               -- (8,4)
                                               -- (8,5)
                                               -- (9,5);
                                              
    \draw[line width=0.6mm, color=orange] (0,0)  -- (0,1);
    \draw[line width=0.6mm, color=orange] (7,4)  -- (8,4) 
                                                 -- (8,5);
    \draw[line width=0.6mm, color=orange] (9,5) -- (10,5);
    
    \draw[dashed] (0,0) -- (10,5);                            
    \node at (7,4) {$\textcolor{teal}{\bigstar}$};
    
\end{tikzpicture}
&
\begin{tikzpicture}[scale = 0.5]
 \newdimen\R
    \R=2.7cm
    
 \draw (30:\R) \foreach \x in {60,90,...,390} {  -- (\x:\R) };  
    \foreach \x/\l/\p in
     {
      60/{\footnotesize  $v_{11}$}/above,
      90/{\footnotesize  $v_{10}$}/above,
      120/{\footnotesize  $v_9$}/above,
      150/{\footnotesize  $v_8$}/left,
      180/{\footnotesize  $v_7$}/left,
      210/{\footnotesize  $v_6$}/left,
      240/{\footnotesize  $v_5$}/below,
      270/{\footnotesize  $v_4$}/below,
      300/{\footnotesize  $v_3$}/below,
      330/{\footnotesize  $v_2$}/right,
      360/{\footnotesize  $v_1$}/right,
      390/{\footnotesize $v_0$}/right
     }
   \node[inner sep=1pt,circle,draw,fill,label={\p:\l}] at (\x:\R) { };

   \draw(60:\R) -- (150:\R);
   \draw(30:\R) -- (180:\R) -- (-30:\R)--(240:\R);
\end{tikzpicture}\\
\end{tabular}
    \caption{On the left, we have the altered 2-Dyck Path of order 5 given by the process in \Cref{thm: rotation of n-gon and shifted Dyck path} and on the right, we have the corresponding rotated 4-angulation of a 12-gon from \Cref{fig:prop-2:etherington-to-dyck-path}.}
    \label{fig:prop-2:etherington-to-dyck-path-rotated}
\end{figure}
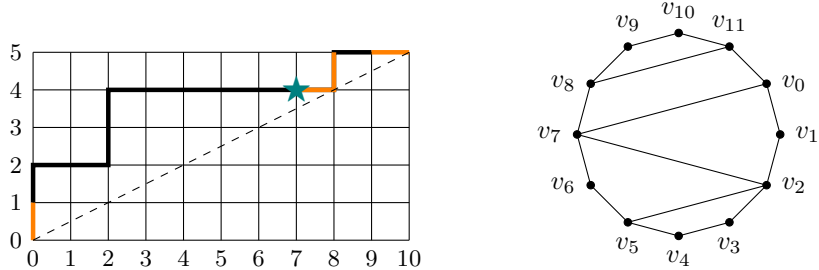

\bibliographystyle{amsalpha}
\bibliography{bibliography.bib}

\end{document}